\documentclass[oneside, book,10 pt]{amsart}
\usepackage{color}
\usepackage{cancel}
\usepackage{amsmath}
\usepackage{amssymb}
\usepackage[normalem]{ulem}
\usepackage{amsfonts}
\usepackage{amsthm}
\usepackage{enumerate}
\usepackage[mathscr]{eucal}
\usepackage{graphicx}
\usepackage{pstricks}
\usepackage[active]{srcltx}
\usepackage{fancyhdr}
\usepackage{verbatim}
\usepackage{fullpage}
\usepackage{hyperref}
\hypersetup{
colorlinks = true, 
urlcolor = blue, 
linkcolor = red, 
citecolor = blue 
}
\usepackage{caption}
\usepackage{float}

\newtheorem{thm}{Theorem}[section]
\newtheorem*{thm*}{Theorem}
\newtheorem{lemma}[thm]{Lemma}
\newtheorem{corollary}[thm]{Corollary}
\newtheorem{prop}[thm]{Proposition}
\newtheorem*{prop*}{Proposition}

\numberwithin{equation}{section}

\title{Generalisation of the Hammersley-Clifford Theorem on Bipartite Graphs}
\author{
Nishant Chandgotia}

\address{
School of Mathematical Sciences\\ Tel Aviv University, Israel}
\email {nishant.chandgotia@gmail.com}

\def\F{{\mathcal F}}
\def\A{{\mathcal A}}
\def\N{\mathbb N}

\newcommand{\Z}{\mathbb{Z}}
\def \B{\mathcal L}

\def \R{\mathbb R}
\def \G{\mathcal G}
\def \V{\mathcal V}
\def \E1{\mathcal E}
\def \M{\mathbf{M}}
\def \Gi{\mathbf{G}}
\def\p{\prime}
\def \H{\mathcal H}
\def \a{\alpha}
\def \b{\beta}
\def \l{\langle}
\def \r{\rangle}
\begin{document}
\maketitle
\begin{abstract} The Hammersley-Clifford theorem states that if the support of a Markov random field has a safe symbol then it is a Gibbs state with some nearest neighbour interaction. In this paper we generalise the theorem with an added condition that the underlying graph is bipartite. Taking inspiration from \cite{brightwell2000gibbs} we introduce a notion of folding for configuration spaces called strong config-folding proving that if all Markov random fields supported on $X$ are Gibbs with some nearest neighbour interaction so are Markov random fields supported on the `strong config-folds' and `strong config-unfolds' of $X$.
\end{abstract}
\section{Introduction}
A Markov random field (MRF) can be viewed as a collection of jointly distributed random variables indexed by the vertices of an undirected graph (denoted by $\G$) satisfying the conditional independence condition: the conditional distributions of the random variables on two finite separated sets are independent given the value of the random variables on the complement of their union. We are interested in determining conditions on the topological support of MRFs such that they are Gibbs states with some nearest neighbour interaction, that is, the distribution of the random variables on a finite set given their value on the outer boundary can be expressed as a normalised product of `weights' associated with patterns on complete subgraphs. The well-known Hammersley-Clifford theorem gives one such condition, a positivity assumption on the MRF given by the presence of a safe symbol in the support, also referred to as the vacuum state. We shall focus on the case where these random variables are finite-valued.

In this paper we view MRFs outside the boundary of safe symbols, folding in the notion of graph folding into our context.
Given a finite undirected graph $\H$ we say that a vertex $a$ can be folded into vertex $b$ if the neighbours of $b$ contain the neighbours of $a$. By removing $a$ from $\H$ we obtain a fold of the graph. A graph is called dismantlable if there is a sequence of folds which leads to a single vertex with or without a self-loop. These notions of folding and dismantlability were introduced by Nowakowski and Winkler in \cite{nowakowskiwinkler} as a characterisation of cop-win graphs.

The presence of folding in $\H$ endows the space of homomorphisms $Hom(\G,\H)$ with some useful properties. Indeed if $a$ can be folded into $b$ in $\H$ then the appearance of $a$ in a homomorphism from $\G$ to $\H$ can be replaced by $b$ to obtain another such homomorphism. However to say that the supports of MRFs are homomorphism spaces is a rather strong assumption. Therefore we abstract some of the properties satisfied by these spaces and introduce a notion of folding in closed configuration spaces $X\subset\A^{\V}$ called strong config-folding where $\A$ is a finite set of symbols corresponding to the vertices of $\H$ and $\V$ is the set of vertices of $\G$ which is assumed bipartite. $X$ is said to have a safe symbol $\star\in\A$ if for all $a\in \A$ and $x\in X$, $\star$ can replace the appearance of $a$ in $x$ to obtain another configuration in that space; in such a case all symbols $a\in \A$ can be strongly config-folded into $\star$. If $X= Hom(\G, \H)$ for some graph $\H$ and bipartite graph $\G$, $a$ can be strongly config-folded into $b$ in $X$ if and only if $a$ can be folded into $b$ in $\H$.

In \cite{brightwell2000gibbs} Brightwell and Winkler established many properties which are preserved under folding and unfolding of graphs. To this we add a `Hammersley-Clifford' property which we will describe next.

A specification is a consistent collection of probability distributions of patterns on finite sets given the pattern on their complement. Every MRF supported on a configuration space $X$ yields a specification on $X$ which is Markovian, that is, the distribution of patterns on the finite sets given the pattern on their complement depends solely on the pattern on their outer boundary. Similarly a Gibbs state with a nearest neighbour interaction on $X$ yields a specification on $X$ which is Gibbsian in nature. An MRF is a Gibbs state with some nearest neighbour interaction if and only if the corresponding Markov specification is a Gibbs specification with the same nearest neighbour interaction. Given a configuration space $X$, in Section 3 of \cite{chandgotiameyerovitch} following \cite{petersen_schmidt1997} we reparametrised specifications on $X$ to obtain Markov and Gibbs cocycles on $X$. The set of these cocycles have a natural vector space structure. Moreover the space of Gibbs cocycles with nearest neighbour interactions is contained in the space of Markov cocycles; the difference of their dimensions measures the extent to which $X$ satisfies the conclusion of the Hammersley-Clifford theorem. Thus the Hammersley-Clifford theorem can be restated in terms of cocycles- if $X$ has a safe symbol then the space of Markov cocycles and Gibbs cocycles with nearest neighbour interactions on $X$ are equal (Theorem \ref{Theorem:Strong_Hammersley-Clifford}).

We call a configuration space $X$ Hammersley-Clifford if the space of Markov cocycles equals the space of Gibbs cocycles with nearest neighbour interactions. Generalising Hammersley-Clifford theorem for bipartite graphs we prove in this paper that the strong config-folds and strong config-unfolds of Hammersley-Clifford spaces are Hammersley-Clifford (Theorems \ref{thm:main theorem} and \ref{thm:G-invariant main theorem}). Further we show that if a configuration space can be strongly config-folded into another then the quotient spaces of their Markov cocycles and Gibbs cocycles with nearest neighbour interactions are isomorphic (Theorem \ref{thm:extension of a Gibbs cocycle} and Corollary \ref{corollary:folds_quotient_isomorphic}). The proof is constructive; the corresponding ``weights" (interactions) can be obtained directly using our proof (Lemma \ref{lemma:construction_of_V^prime}). We also obtain versions of the theorem when the space of cocycles is invariant under a subgroup of automorphisms of the graph.

There has been some work regarding conditions under which the conclusion of the Hammersley-Clifford theorem holds and some examples where it does not: J. Moussouris provided examples of MRFs on a finite graph which are not Gibbs states with any nearest neighbour interaction \cite{mouss}. When the underlying graph is finite, there are algebraic conditions on the support \cite{sturmfel} where the conclusion of the Hammersley-Clifford theorem holds. In \cite{laur} Lauritzen proved that when the underlying graph is finite and decomposable then every MRF is a Gibbs state with some nearest neighbour interaction. This is essentially in contrast with our work; a bipartite graph is decomposable if and only if it is a forest. If the underlying graph is $\mathbb{Z}$ and the MRF is shift-invariant then the conclusion of the Hammersley-Clifford theorem holds without any further assumptions \cite{Markovfieldchain}. Furthermore, in that setting any MRF is a stationary Markov chain. When $\A$ is a general measure space, Theorems 10.25 and 10.35 in \cite{Georgii} provide certain mixing conditions which guarantee the conclusion as well.
Even when the underlying graph is $\mathbb{Z}$, this conclusion can fail for countable $\mathcal{A}$  \cite[Theorem 11.33]{Georgii}, or if we drop the assumption of shift-invariance \cite{dob}.
 When the underlying graph is $\mathbb{Z}^d$ and $d>1$, the conclusion fails even in the shift-invariant and finite alphabet case \cite[Chapter 5]{Chandgotia} and \cite[Section 9]{chandgotiameyerovitch}. Let $C_n$ denote the $n$-cycle. The space of Markov and Gibbs cocycles on $Hom(\Z^d, C_n)$ was analysed in \cite{chandgotiameyerovitch} for $d\geq 2$ and $n\neq 4$. Here we show that $Hom(\G, C_4)$ is Hammersley-Clifford for any bipartite graph $\G$. 

Section \ref{section:Background and Notation} begins with well-known notions related to this work e.g. MRFs and Gibbs States, invariance under group actions and the Hammersley-Clifford theorem. In Subsection \ref{subsection:Hammersley-Clifford} we take inspiration from symbolic dynamics to define n.n.constraint spaces. In Subsection \ref{subsection:cocycles} we introduce Markov and Gibbs cocycles and their relationship to the Hammersley-Clifford theorem. Section \ref{section:Hamm-cliff space} builds up the necessary background for this work. In Subsection \ref{subsection:Hamm-cliff spaces} we introduce Hammersley-Clifford spaces and in Subsection \ref{subsection:Markov-similar} we introduce Markov-similarity and $V$-good pairs. In Subsection \ref{subsection:Graphfolding} we introduce strong config-folding. Section \ref{section: main results} states and proves the main results of this paper. Since the proofs are technical we work out a concrete example of our results in Subsection \ref{subsection:concrete}. Finally we conclude with some further questions in Section \ref{section:conclude}.

A small note on the notation: The calligraphic letters $\G$ and $\H$ denote graphs; $\A$ denotes a finite set (the alphabet) and the letters in bold font $\M$, $\Gi$ represent the space of cocycles. $\V$ denotes the set of vertices of $\G$. Among the symbols in regular font $X$ represents a closed space of configuration. $Gr$ denotes a subgroup of automorphisms of the graph $\G$. $x, y, z$ denote configurations while $A$, $F$ denote subsets of $\V$ and $u$, $v$, $w$ the vertices of $\G$. $a, b, c$ elements of the alphabet $\A$. In most cases $a$ can be strongly config-folded into the symbol $b$ and $V$ is an interaction. The greek letters $\alpha, \b$ denote patterns.
\section{Background and Notation} \label{section:Background and Notation}
\subsection{MRFs}
\large
By a graph $\G=(\V, \E1)$ we mean a countable locally finite undirected graph without any self-loops or multiple edges. The adjacency relation is denoted by $\sim_\G$. Given any set $F\subset \V$ the \emph{boundary of the set $F$} is defined as the set of vertices outside the set $F$ which are adjacent to $F$, that is,
$$\partial F:=\{u\in \V\setminus F\:|\:\text{ there exists }v\in F \text{ s.t. } u \sim_\G v\}.$$
Sometimes this is also called the external vertex boundary of the set $F$.

Given a finite set $\A$, $\A^\V$ is a compact topological space under the product topology. For any finite set $F\subset \V$ and $\alpha\in \A^F$ we denote by $[\alpha]_F$ the cylinder set
$$[\alpha]_F:=\{x\in \A^\V\:|\:x|_F=\alpha\}.$$
Similarly given $x\in \A^\V$ and $F\subset \V$, $[x]_F$ denotes the cylinder set $[x|_F]_F$.
The collection of cylinder sets generate the Borel $\sigma$-algebra on $\A^\V$. The set $\A$ will be referred to as an \emph{alphabet} with finitely many \emph{symbols} which when placed on vertices of the graph $\G$ yield \emph{configurations}, that is, elements of $\A^\V$ and \emph{patterns}, that is, elements of $\A^F$ for some set $F\subset \V$.

An \emph{MRF} is a Borel probability measure $\mu$ on $\A^\V$ with the property that for all finite sets $A, B\subset \V$ such that $\partial A\subset B\subset A^c$ and $\alpha\in \A^A, \beta\in \A^B$ satisfying $\mu([\beta]_B)>0$
$$\mu([\alpha]_A\:|\:[\beta]_B)= \mu([\alpha]_A\:|\:[\beta]_{\partial A}).$$
An equivalent definition is the following: If $x$ is a configuration chosen randomly according to the measure $\mu$, and $A, B\subset\V$ are finite \emph{separated }sets in $\G$ (meaning that $A$ and $B$ are disjoint and $u\nsim_\G v$ for all $u\in A$ and $v\in B$), then conditioned on $x|_{\V\setminus (A\cup B)}$, $x|_{A}$ and $x|_B$ are independent random variables. Here we restrict our attention to boundaries of thickness $1$. In general thicker boundaries can also be considered for similar notions.

A stronger notion of an MRF obtained by requiring this conditional independence for all sets $A, B\subset \V$ which are separated in $\G$ (finite or not) is called a \emph{global MRF}. This paper is concerned with the former notion of independence, where both $A$ and $B$ are assumed to be finite.

\subsection{Gibbs States with Nearest Neighbour Interactions}
Let $d_\G$ denote the graph distance on $\G$. Given any finite set $A\subset \V$ let $diam(A)$ denote the \emph{diameter of the set $A$} defined by
$$diam(A):= \max_{u, v\in A}d_{\G}(u,v).$$
Given a closed configuration space $X\subset \A^\V$ and $F\subset \V$, denote by $\B_F(X)$ \emph{the language of $X$ on $F$} defined as the set of allowed patterns on $F$, that is,
$$\B_F(X):=\{\alpha\in \A^F\:|\: \text{ there exists }x\in X\text{ s.t. }x|_{F}=\alpha\}.$$
Note that $\B_\V(X)=X$. Denote by $\B(X)$ \emph{the language of $X$} defined as the set of all allowed patterns on finite sets, that is,
$$\B(X):= \bigcup_{F\subset \V\text{ finite}}\B_F(X).$$
From the following lemma we see that the language completely describes the configuration space. 
\begin{prop}\label{Proposition:language gives all}
Let $\A$ be a finite set, $\G=(\V, \E1)$ be a graph and $X, Y\subset \A^{\V}$ be closed sets. Then $X\subset Y$ if and only if $\B_A(X)\subset \B_A(Y)$ for all $A\subset \V$ finite.
\end{prop}

By definition if $X\subset Y$ then $\B_A(X)\subset \B_A(Y)$ for all finite sets $A\subset \V$. The converse is true because $X$ and $Y$ are closed.

A closed configuration space relevant to us is the \emph{support} of a probability measure $\mu$ denoted by $supp(\mu)$ and defined as the intersection of all closed sets $X\subset \A^\V$ with full measure.

An interaction on $X$ is a real-valued function on the language, $V: \B(X)\longrightarrow \R$ satisfying certain summability conditions. A \emph{nearest neighbour interaction} is an interaction $V$ on $X$ such that it is supported on patterns on cliques (complete subgraphs) of $\G$, that is, $V(\alpha)=0$ for all patterns $\a\in \B_F(X)$ where $diam(F)>1$. If the underlying graph $\G$ is bipartite then a nearest neighbour interaction is an interaction supported on patterns on edges and vertices. We will denote by $\l a,b\r_{\{v,w\}}$ the pattern in $\A^{\{v,w\}}$ given by
$${\l a,b\r_{\{v,w\}}}(v):=a\text{ and }{\l a,b\r_{\{v,w\}}}(w):=b$$
and by $\l a\r_v$ the pattern in $\A^{\{v\}}$ given by $\l a\r_v(v):=a$.

A\emph{ Gibbs state with a nearest neighbour interaction $V$} is an MRF $\mu$ such that for all $x\in supp(\mu)$ and $A, B\subset \V$ finite satisfying $\partial A\subset B\subset A^c$
$$\mu([x]_A\:|\:[x]_B):= \frac{\prod_{C\subset A\cup \partial A}e^{V(x|_C)}}{Z_{A, x|_{\partial A}}}$$
where $Z_{A, x|_{\partial A}}$ is the uniquely determined normalising factor dependent upon $A$ and $x|_{\partial A}$ so that $\mu(X)=1$.

Note that Gibbs states with nearest neighbour interactions and MRFs can be distinguished by conditional distributions mentioned in the equation above. In Subsection \ref{subsection:cocycles} we will consider a parameterisation of the space of conditional probability distributions to formally study the distinction at that level. Also note that by this definition of Gibbs states the constraints on the support are extrinsic; there is an intrinsic way of constraining the support by allowing the interactions to be infinite. This leads to a different notion of Gibbs states which we will not pursue.

This paper is concerned with conditions on the support of MRFs, which imply that they are Gibbs with some nearest neighbour interaction.

\subsection{Invariant Spaces, Measures and Interactions}
An \emph{automorphism} of the graph $\G$ is a bijection on the vertex set $g:\V\longrightarrow \V$ which preserves the adjacencies, that is, $u\sim_\G v$ if and only if $gu\sim_\G gv$. Let the group of all automorphisms of the graph $\G$ be denoted by $Aut(\G)$.

There is a natural action of $Aut(\G)$ on patterns and configurations: given $\a\in \A^F$, $x\in \A^\V$ and $g\in Aut(\G)$ we have $g\a\in \A^{gF}$ and $gx\in \A^\V$ given by
$$(g\a)_{gv}:= \a_v\text{ and}$$
$$(gx)_v:= x_{g^{-1}v}.$$
This induces an action on measures on the space $\A^\V$ given by
$$(g\mu)(L):=\mu(g^{-1}L)$$
for all measurable sets $L\subset \A^\V$.

For a given subgroup $Gr\subset Aut(\G)$, a set of configurations $X\subset \A^\V$ is said to be \emph{$Gr$-invariant} if $gX=X$ for all automorphisms $g\in Gr$. Similarly a measure $\mu$ on $\A^\V$ is said to be $Gr$-invariant if $g\mu= \mu$ for all $g\in Gr$. Note, for any subgroup $Gr\subset Aut(\G)$, if $\mu$ is a $Gr$-invariant probability measure then $supp(\mu)$ is also a $Gr$-invariant configuration space.
If $\G=\Z$ and $Gr$ is the group of translations of $\Z$, then $Gr$-invariant closed spaces of configurations in $\A^\Z$ are precisely the shift spaces \cite[Theorem 6.1.21]{LM} and $Gr$-invariant probability measures correspond to stationary stochastic processes on the $\Z$ lattice.

Let $X\subset \A^\V$ be a closed configuration space invariant under a subgroup $Gr\subset Aut(\G)$. Then $Gr$ acts on the interactions on $X$: Given an interaction $V$ on $X$ for all $\a\in \A^F$ and $g\in Gr$
$$gV(\a):= V(g^{-1}\a).$$

\subsection{Hammersley-Clifford Theorem and the Support of MRFs} \label{subsection:Hammersley-Clifford}As in \cite{Markovfieldchain,chandgotiameyerovitch}, closed subsets $X\subset \A^\V$ will be called \emph{topological Markov fields} if for all $x,y\in X$ and finite $F\subset \V$ satisfying $x|_{\partial F}= y|_{\partial F}$ there exists $z\in X$ such that
$$z_v=\begin{cases}
x_v&\text{ if }v\in F\\
y_v&\text{ if }v\in \V\setminus F.
\end{cases}•$$
The support of every MRF is a topological Markov field. If the underlying graph is finite then further $X\subset \A^\V$ is the support of an MRF if and only if it is a topological Markov field. If $Gr$ is the group of translations of the $\Z$ lattice then $X\subset \A^\Z$ is the support of a $Gr$-invariant MRF if and only it is a non-wandering (a certain irreducibility condition) $Gr$-invariant n.n.constraint space (also known as nearest neighbour shifts of finite type). However in general characterising the support of an MRF seems to be a hard question. This is not even known in case the graph is $\Z^2$ for MRFs invariant under translations \cite[Section 10]{chandgotiameyerovitch}.

A closed configuration space $X\subset \A^\V$ is said to have a safe symbol $\star$ if for all $A\subset \V$ and $x\in X$ we can `legally' replace the symbols on $A$ by $\star$, that is, there exists $y\in X$ satisfying
\begin{equation*}
y_{v}:=\begin{cases}
x_v\text{ if }v\in A\\
\star\:\:\text{ if }v\in A^c.
\end{cases}•
\end{equation*}•
We will now state the Hammersley-Clifford theorem.

\begin{thm}[Hammersley-Clifford Theorem, weak version \cite{HammersleyClifford71,Preston,Chandgotia}]\label{Theorem:Weak_Hammersley-Clifford}
Let $\G=(\V, \E1)$ be a graph, $X$ be a configuration space with a safe symbol on $\G$ and $\mu$ be an MRF such that $supp(\mu)=X$. Then
\begin{enumerate}
\item The measure $\mu$ is Gibbs for some nearest neighbour interaction.
\item If $\mu$ is a $Gr$-invariant MRF for some subgroup $Gr\subset Aut(\G)$ then $\mu$ is a Gibbs state with some $Gr$-invariant nearest neighbour interaction.
\end{enumerate}
\end{thm}

This theorem led to the sparkling of our interest in the field: the study of conditions on the support of MRFs which imply that they are Gibbs. We will now prove that the support of the measures mentioned in Theorem \ref{Theorem:Weak_Hammersley-Clifford} have some ``combinatorial structure".

The following definitions take inspiration from symbolic dynamics (\cite{LM}). Let $\F$ be a given set of patterns on finite sets. Then the configuration space with constraints $\F$ is defined to be
\begin{eqnarray*}
X_\F:=\{x\in \A^\V\:|\text{ patterns from }\F\text{ do not appear in }x\}.
\end{eqnarray*}•
A set of constraints $\F$ is called \emph{nearest neighbour} if $\F$ consists of patterns on cliques, that is, for all $\alpha\in \F\cap\A^F$, $diam(F)\leq 1$.

A \emph{n.n.constraint space} is a configuration space with nearest neighbour constraints. Note that if $\G$ is bipartite then $\F$ consists of patterns on edges and vertices. These spaces correspond to nearest neighbour shifts of finite type which are replete in the sphere of symbolic dynamics.

\noindent\textbf{Examples: }
\begin{enumerate}
\item\emph{(The hard core model)} Here the alphabet $\A:=\{0,1\}$ and the constraint set is given by
$$\F:=\{\l 1,1\r_{\{u,v\}}\:|\:u\sim_\G v\}.$$
This constrains the configurations so that symbols on adjacent vertices cannot both be $1$.
\item\emph{(The space of $3$-colourings)} Here the alphabet $\A:=\{0,1,2\}$ and the constraint set is given by
$$\F:=\{\l a,a\r_{\{v,w\}}\:|\: v\sim_\G w \text{ and }a\in \{0,1,2\}\}.$$
This constrains the configurations so that symbols on adjacent vertices are distinct.
\end{enumerate}

Note that the n.n.constraint spaces given above are of a very special class, namely the constraints on all edges of the graph $\G$ are the same. These configuration spaces correspond to homomorphism spaces defined as the following: Given an undirected graph $\H=(\V_\H, \E1_\H)$ without multiple edges \emph{a homomorphism} from $\G$ to $\H$ is a map $x:\V\longrightarrow \V_\H$ such that for all $v\sim_\G w$, $x_v\sim_\H x_w$. \emph{The space of all homomorphisms }from $\G$ to $\H$ is denoted by $Hom(\G, \H)$. For instance the hard core model is the space $Hom(\G, \H)$ where $\H$ is given by Figure \ref{figure:hardsquaregraph}
\begin{figure}[h]
\includegraphics[angle=0,
width=.2\textwidth]{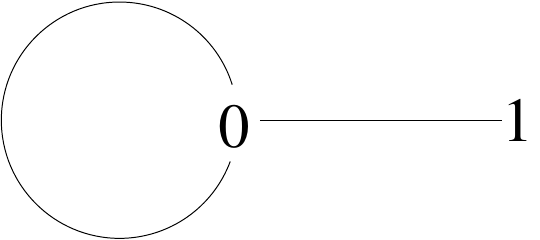}
\caption{Graph for the Hard Core Model \label{figure:hardsquaregraph}}
\end{figure}
and the space of $3$-colourings is $Hom(\G, C_3)$ where $C_3$ is the $3$-cycle with vertices $0$, $1$ and $2$. Also note that the hard core model has a safe symbol $0$ but the space of $3$-colourings does not have any safe symbol.
Given graphs $\G$ and $\H$, $Hom(\G,\H)$ is an n.n.constraint space where the constraint is given by 
$$\F:=\{\l a,b\r_{\{v,w\}}\:|\: a\nsim_\H b\in \V_\H\text{ and }v\sim_\G w\}.$$

\noindent Then for all $x\in X_\F$ and vertices $v\sim_\G w$, $x_v\sim_\H x_w$ which implies $x\in Hom(\G, \H)$. Conversely for all homomorphisms $x\in Hom(\G,\H)$ and vertices $v\sim_\G w$ we have $x|_{\{v,w\}}\notin \F$ and hence $x\in X_\F$.

N.N.Constraint spaces arise naturally in the study of MRFs as is shown in the following propositions.
\begin{prop}\label{Proposition: nearest neighbour is TMF}
Every n.n.constraint space is also a topological Markov field.
\end{prop}
\begin{proof}
Let $\G=(\V, \E1)$ be a graph, $\A$ be a finite set and $X\subset \A^\V$ be an n.n.constraint space on $\G$. Consider $A\subset \V$ finite and $x, y\in X$ such that $x|_{\partial A}= y|_{\partial A}$. We want to prove that $z\in \A^\V$ defined by
\begin{equation*}
z_v:=
\begin{cases}
x_v&\text{ if }v\in A\cup \partial A\\
y_v&\text{ if }v\in A^c
\end{cases}
\end{equation*}•
is an element of $X$. Let $B\subset \V$ be a clique. If $B\cap A\neq \emptyset$ then $B\subset A\cup \partial A$ and $z|_{B}= x|_B\in \B_B(X)$ else $B\cap A= \emptyset$ implying $z|_{B}= y|_B\in \B_B(X)$. Since $X$ is an n.n.constraint space $z\in X$.
\end{proof}
The following proposition gives a partial converse.
\begin{prop}\label{Proposition: TMF is nearest neighbour}
Every topological Markov field with a safe symbol is also an n.n.constraint space.
\end{prop}
\noindent\textbf{Remark:} If $\mu$ is an MRF then $supp(\mu)$ is a topological Markov field. Thus this proposition implies that if a measure $\mu$ satisfies the hypothesis of the weak Hammersley-Clifford theorem (Theorem \ref{Theorem:Weak_Hammersley-Clifford}), that is, if $\mu$ is an MRF such that $supp(\mu)$ has a safe symbol then $supp(\mu)$ is an n.n.constraint space. The conclusion of this proposition does not hold without assuming presence of a safe symbol (comments following proof of Proposition 3.5 in \cite{Markovfieldchain}).
\begin{proof}
Let $\A$ be a finite set, $\G=(\V, \E1)$ be a given graph and $X\subset \A^\V$ be a topological Markov field on the graph $\G$ with a safe symbol. Let $\star$ be a safe symbol for $X$. Consider the set
\begin{eqnarray*}
\F:=\{\a\in \A^A\:|\: A\subset \V \text{ forms a clique and there does not exist }x\in X \text{ such that } x|_A=\a\}.
\end{eqnarray*}•
Note that $ X\subset X_\F$ and if $A\subset \V$ is a clique then $\B_A(X_\F)= \B_A(X)$. We want to prove that $X_\F\subset X$. We will proceed by induction on $n\in \N$, the hypothesis being: For all $A\subset \V$ such that $|A|=n$, $\B_A(X_\F)\subset \B_A(X)$.

The base case follows immediately. Suppose for some $n\in \N$, for all $A\subset \V$ satisfying $|A|\leq n$, $\B_A(X_\F)\subset \B_A(X)$.

For the induction step consider $A\subset \V$ such that $|A|=n+1$. There are two cases to consider: If $A$ is a clique then $\B_A(X_\F)=\B_A(X)$. If $A$ is not a clique then there exists $v\in A$ such that $|\partial\{v\}\cap A|<n$. Let $\a\in \B_A(X_\F)$. We will prove that $\a\in \B_A(X)$. Now $|\left(\{v\}\cup \partial\{v\}\right)\cap A|, |A\setminus \{v\}|\leq n$, thus the induction hypothesis implies
$$\a|_{(\{v\}\cup \partial\{v\})\cap A}\in \B_{(\{v\}\cup \partial\{v\})\cap A}(X)$$
and
$$\a|_{A\setminus \{v\}}\in \B_{A\setminus \{v\}}(X).$$
Consider $x,y\in X$ such that
$$x|_{(\{v\}\cup \partial\{v\})\cap A}= \a|_{(\{v\}\cup \partial\{v\})\cap A}$$
and
$$y|_{A\setminus \{v\}}= \a|_{A\setminus \{v\}}.$$
Since $\star$ is a safe symbol for $X$ therefore
$x^\star, y^\star\in \A^\V$ given by
$$x^\star_w:=\begin{cases}
x_w&\text{if }w\in (\{v\}\cup \partial\{v\})\cap A\\
\star&\text{otherwise}
\end{cases}•$$
and
$$y^\star_w:=\begin{cases}
y_w&\text{ if }w\in A\setminus\{v\}\\
\star&\text{otherwise}
\end{cases}•$$
are configurations in $X$. Note that $x^\star_w=x_w= \a_w$, $y^\star_w =y_w= \a_w$ if $w\in \partial\{v\}\cap A$ and $x^\star_w= y^\star_w= \star$ if $w\in A^c$. Therefore $x^\star|_{ \partial\{v\}}= y^\star|_{ \partial\{v\}}$. Since $X$ is a topological Markov field, $z\in \A^\V$ defined by
\begin{equation*}
z_w:=
\begin{cases}
x^\star_w&\text{ if }w\in \{v\}\cup \partial \{v\}\\
y^\star_w&\text{otherwise}
\end{cases}•
\end{equation*}•
is an element of $X$. But $z_v= x^\star_v= x_v= \a_v$ and $z_w= y^\star_w= y_w= \a_w$ if $w\in A\setminus \{v\}$. Hence $z|_A=\a\in \B_A(X)$. This completes the induction. By Proposition \ref{Proposition:language gives all} $X_{\F}\subset X$. Hence $X= X_\F$.
\end{proof}

N.N.Constraint spaces allow us to change configurations one site at a time provided the edge-constraints are satisfied. To state this rigorously we define the following: given $x\in \A^\V$, and distinct vertices $w_1, w_2, \ldots, w_r \in\V$ and $c_1, c_2, \ldots, c_r\in \A$ we denote by $\theta^{w_1, w_2,\ldots, w_r}_{c_1, c_2,\ldots,c_r}(x)$ an element of $\A^\V$ given by
\begin{equation*}
(\theta^{w_1,w_2, \ldots, w_r}_{c_1, c_2,\ldots, c_r}(x))_u:=
\begin{cases}
x_u&\text{ if }u\neq w_1, w_2\ldots, w_r\\
c_i&\text{ if }u=w_i\text{ for some }1\leq i \leq r.
\end{cases}
\end{equation*}•

\begin{prop}\label{prop:changing_remain_nnconstraint}
Let $\A$ be a finite set, $\G=(\V, \E1)$ be a bipartite graph, $X\subset \A^\V$ be an n.n.constraint space on the graph $\G$ and $x\in X$. Let $w_1, w_2, \ldots, w_r \in\V$ be distinct vertices such that $w_i \nsim_\G w_j$ for $1\leq i,j \leq r$ and $c_1, c_2, \ldots, c_r\in \A$ such that $\l c_i, x_{w^\prime}\r_{\{w_i, w^\prime\}}\in \B_{\{w_i, w^\prime\}}(X)$ for all $w^\prime\sim_\G w_i$ and $1\leq i \leq r$. Then $\theta^{w_1, w_2, \ldots, w_r}_{c_1, c_2\ldots, c_r}(x)\in X$.
\end{prop}
Specialising to $r=1$, if $X\subset \A^\V$ is an n.n.constraint space and $x\in X$ then for $v\in \V$ and $c\in \A$, $\theta^v_c(x)\in X$ if and only if $\l x_w,c\r_{\{w,v\}}\in \B_{\{w,v\}}(X)$ for all $w\sim_\G v$.
\begin{proof}
The constraint set for $X$ consists only of patterns on edges and vertices. Thus it is sufficient to check for all $v\sim_\G w$ that
$$\theta^{w_1, w_2, \ldots, w_r}_{c_1, c_2\ldots, c_r}(x)\vert_{\{v,w\}}\in \B_{\{v,w\}}(X).$$

Since $w_i\nsim_\G w_j$ for all $1\leq i,j\leq r$ at most one among $v$ and $w$ is $w_i$ for some $1\leq i\leq r$. If both of them are not equal to $w_i$ then
$$\theta^{w_1, w_2, \ldots, w_r}_{c_1, c_2\ldots, c_r}(x)|_{\{v,w\}}=x|_{\{v,w\}}\in\B_{\{v,w\}}(X). $$
Otherwise we may assume $v=w_i$ for some $1\leq i \leq r$ giving us
$$\theta^{w_1, w_2, \ldots, w_r}_{c_1, c_2\ldots, c_r}(x)|_{\{v,w\}}=\l c_i,\ x_w\r_{\{w_i,w\}}\in\B_{\{w_i,w\}}(X). $$
\end{proof}

\subsection{The Asymptotic Relation, Cocycles and the Strong Version of the Hammersley-Clifford Theorem} \label{subsection:cocycles}
This subsection shall closely follow \cite[Section 3]{chandgotiameyerovitch}. Given a closed configuration space $X\subset \A^\V$ the set of \emph{asymptotic pairs} is given by
$$\Delta_X:=\{(x,y)\in X\times X\:|\: x,y \text{ differ at finitely many sites}\}.$$
In the case when the graph $\G= \Z^d$ and $X$ is a configuration space invariant under translation, the asymptotic relation coincides with the \emph{homoclinic relation}. If $\G$ is finite then $\Delta_X=X\times X$.
Following \cite{brightwell2000gibbs} a space $X$ is called \emph{frozen} if
$$\Delta_X=\{(x,x)\:|\:x\in X\}.$$

As in \cite{chandgotiameyerovitch} we shall now parametrise the space of conditional probabilities. A (real-valued) \emph{$\Delta_X$-cocycle} is a function $M:\Delta_X\longrightarrow \R$ satisfying
\begin{eqnarray*}\label{equation:cocycle_condition}
M(x,z)=M(x,y)+M(y,z) \text{ whenever }(x,y),(y,z)\in \Delta_X.
\end{eqnarray*}
If $M$ is a $\Delta_X$-cocycle then for all pairs $(x,y)\in \Delta_X$ 
\begin{eqnarray}
M(x,y)+M(y,x)=M(x,x)=M(x,x)+M(x,x)=0.\label{equation:flip_gives_negative_same_0}
\end{eqnarray}
A $\Delta_X$-cocycle is called a \emph{Markov cocycle} if in addition for any $(x, y)\in \Delta_X$, the value $M(x,y)$ depends only upon patterns on vertices where $x$ and $y$ differ and its boundary, that is, if $F$ is the set of vertices where $x$ and $y$ differ then $M(x,y)$ depends only of $x|_{F\cup \partial F}$ and $y|_{F\cup \partial F}$.

Given a subgroup $Gr\subset Aut(\G)$, a $Gr$-invariant $\Delta_X$-cocycle is a $\Delta_X$-cocycle $M$ which satisfies
$$M(x,y)=M(gx, gy)$$
for all $(x,y)\in \Delta_X$ and $g\in Gr$.

Any MRF $\mu$ yields a Markov cocycle $M$ on $supp(\mu)$ by
\begin{equation}\label{equation:specification_to_cocycle}
M(x,y):= \log\left(\frac{\mu([y]_\Lambda)}{\mu([x]_{\Lambda})}\right)\text{ for all }(x,y)\in \Delta_{supp(\mu)}
\end{equation}
for any $\Lambda\supset F\cup \partial F$ where $F$ is the set of vertices where $x$ and $y$ differ. Since $\mu$ is an MRF the right hand side is independent of the choice of $\Lambda$. For example the uniform MRF (where conditioned on the pattern on the boundary of a finite set $F$, all patterns on $F$ are equiprobable) yields the Markov cocycle $M=0$ and if the graph $\G$ is finite then any MRF $\mu$ yields the cocycle
\begin{equation*}
M(x,y):= \log\left(\frac{\mu(y)}{\mu(x)}\right)\text{ for all }x,y\in supp(\mu).
\end{equation*}
The function $\rho:\Delta_{supp(\mu)}\longrightarrow \R_+$ given by $\rho(x,y)=e^{M(x,y)}$ is the $\Delta_X$-Radon-Nikodym cocycle of $\mu$ as in \cite{petersen_schmidt1997}. This correspondence can be further generalised; given a topological Markov field we can consider a system of consistent conditional probability distributions with the Markov property called Markov specifications. There is a bijective correspondence between the space of Markov cocycles and Markov specifications. For a more detailed discussion on this topic see \cite{chandgotiameyerovitch}.

Given a topological Markov field $X$, \emph{the Gibbs cocycle }on $X$ corresponding to an interaction $V$ is a $\Delta_X$-cocycle given by
$$M(x,y):=\sum_{A\subset \V \text{ finite}}V(y|_A)-V(x|_A) \text{ for all } (x,y)\in \Delta_X.$$

Note that the sum is finite since there are only finitely many non-zero terms whenever $(x,y)\in \Delta_X$. Evidently any Gibbs cocycle with a nearest neighbour interaction is a Markov cocycle. This corresponds to the fact that every Gibbs state with a nearest neighbour interaction is an MRF.

Thus the distinction between MRFs and Gibbs state with a nearest neighbour interaction on the level of measures naturally yields a distinction on the level of corresponding cocycles.

\begin{prop}
Let $\mu$ be an MRF and $M$ be a Markov cocycle on $supp(\mu)$ given by
\begin{equation*}
M(x,y)= \log\left(\frac{\mu([y]_\Lambda)}{\mu([x]_{\Lambda})}\right)\text{ for all }(x,y)\in \Delta_{supp(\mu)}
\end{equation*}
for any $\Lambda\supset F\cup \partial F$ where $F$ is the set of vertices where $x$ and $y$ differ. Then $\mu$ is a Gibbs state with a nearest neighbour interaction if and only if $M$ is a Gibbs cocycle with some nearest neighbour interaction.
\end{prop}
\noindent The proof of the proposition follows from the discussions preceding the proposition.

Let $X$ be a topological Markov field. We shall denote the set of all Markov cocycles by $\M_X$ and the set of all Gibbs cocycles with nearest neighbour interactions by $\Gi_X$ . Given a subgroup $Gr\subset Aut(\G)$ we denote by $\M^{Gr}_X$ the set of all $Gr$-invariant Markov cocycles and by $\Gi^{Gr}_X$ the space of all Gibbs cocycles with $Gr$-invariant nearest neighbour interactions. Note that the space of $Gr$-invariant Gibbs cocycles with nearest neighbour interactions is not always the same as the space of Gibbs cocycles with $Gr$-invariant nearest neighbour interactions. An example can be found in \cite[Section 5]{chandgotiameyerovitch}.

The space of Markov cocycles has a natural vector space structure. Indeed given $M_1, M_2\in \M_X$ and $c\in \R$
$$cM_1+ M_2\in \M_X$$
where the addition is point-wise, that is, $(cM_1+M_2)(x,y)= cM_1(x,y)+M_2(x,y)$ for all $(x,y)\in \Delta_X$. Further it follows that  given a subgroup $Gr\subset Aut(\G)$, $\Gi_X$ and $\Gi^{Gr}_X\subset \M^{Gr}_X$ are subspaces of $\M_X$. If $\G$ is a finite graph the conditions under which $\M_X= \Gi_X$ are very similar to the balanced conditions as mentioned in \cite{mouss}.

A close inspection of the proof of the weak version of the Hammersley-Clifford theorem (Theorem \ref{Theorem:Weak_Hammersley-Clifford}) yields another formulation in terms of cocycles. We will not use this or the weaker version for proving the main results of this paper (Theorems \ref{thm:main theorem} and \ref{thm:G-invariant main theorem}). 
\begin{thm}[Hammersley-Clifford, strong version] \label{Theorem:Strong_Hammersley-Clifford} Let $\G=(\V, \E1)$ be a graph and $X$ be a topological Markov field on the graph $\G$ with a safe symbol. Then,
\begin{enumerate}
\item
Any Markov cocycle on $X$ is a Gibbs cocycle with a nearest neighbour interaction, that is, $\M_X=\Gi_X$.
\item
Given a subgroup $Gr\subset Aut(\G)$ every $Gr$-invariant Markov cocycle on $X$ is a Gibbs cocycle with some $Gr$-invariant nearest neighbour interaction, that is, $\M^{Gr}_X=\Gi^{Gr}_X$.
\end{enumerate}•
\end{thm}

Given a topological Markov field $X$ with a safe symbol, any MRF $\mu$ such that the $supp(\mu)=X$ yields by $\eqref{equation:specification_to_cocycle}$ a Markov cocycle on $X$. Moreover if $\mu$ is invariant under a subgroup $Gr\subset Aut(\G)$ then the cocycle obtained is also invariant under the same. By Theorem \ref{Theorem:Strong_Hammersley-Clifford} the cocycle is Gibbs with a $Gr$-invariant nearest neighbour interaction. Thus we know that the measure is Gibbs with some $Gr$-invariant nearest neighbour interaction. Hence Theorem \ref{Theorem:Strong_Hammersley-Clifford} generalises Theorem \ref{Theorem:Weak_Hammersley-Clifford}. However the proof of the first part of this version follows from Theorem \ref{Theorem:Weak_Hammersley-Clifford} with the additional knowledge that given a Markov cocycle on a topological Markov field $X$ with a safe symbol there exists a corresponding MRF $\mu$ such that $supp(\mu)=X$. This in turn is implied by arguments very similar to those in the proof of Proposition 3.3 in \cite{chandgotiameyerovitch}. The second part of the theorem can be proved using Theorem 2.0.6 in \cite{Chandgotia}, noting that the conclusion holds even if the MRF is not invariant under $Gr$ but the corresponding Markov cocycle is.

We seek a generalisation of Theorem \ref{Theorem:Strong_Hammersley-Clifford} when the graph $\G$ is bipartite.

\section{Hammersley-Clifford Spaces and Strong Config-Foldings}
\label{section:Hamm-cliff space}
\subsection{Hammersley-Clifford Spaces}
\label{subsection:Hamm-cliff spaces}
A topological Markov field $X\subset \A^\V$ is called \emph{Hammersley-Clifford} if the space of Markov cocycles on $X$ is equal to the space of Gibbs cocycles on $X$, that is, $\M_X= \Gi_X$. If $X$ is invariant under the some subgroup $Gr\subset Aut(\G)$ then $X$ is said to be \emph{$Gr$-Hammersley-Clifford} if $\M^{Gr}_X= \Gi^{Gr}_X$.

\noindent\textbf{Examples:}
\begin{enumerate}
\item
A frozen space of configurations. \label{enumerate:frozen_configurations}

If $X$ is frozen then $\Delta_X$ is the diagonal relation. Then $M\equiv 0$ is the only Markov cocycle on the space. It is Gibbs for the interaction $V\equiv 0$.
\item
A topological Markov field with a safe symbol.\label{enumerate:tmf_safe_symbol}

Theorem \ref{Theorem:Strong_Hammersley-Clifford} implies that any $Gr$-invariant configuration space with a safe symbol is $Gr$-Hammersley-Clifford for any subgroup $Gr\subset Aut(\G)$.

\item
$Hom(\G,Edge)$ where $Edge$ consists of two vertices $0$ and $1$ connected by a single edge. \label{enumerate:homomorphisms_edge}

If $\G$ is not bipartite then $Hom(\G, Edge)$ is empty. If $\G$ is bipartite and connected, then $Hom(\G, Edge)$ consists of two configurations only. It follows that $Hom(\G, Edge)$ is $Gr$-Hammersley-Clifford for any subgroup $Gr\subset Aut(\G)$ and graph $\G$.

\item
$Hom(\Z^d, C_n)$ where $C_n$ is an n-cycle, $d>1$ and $n\neq 4$ \label{enumerate:homomorphisms_cycle}\cite{chandgotiameyerovitch}.

This gives examples of Hammersley-Clifford spaces which are not $Gr$-Hammersley-Clifford spaces for some subgroup $Gr\subset Aut(\Z^d)$. It will follow from Theorem \ref{thm:G-invariant main theorem} below and Example \ref{enumerate:homomorphisms_edge} above that $Hom(\G, C_4)$ is both Hammersley-Clifford and $Gr$-Hammersley-Clifford for all bipartite graphs $\G$ and subgroups $Gr\subset Aut(\G)$.
\end{enumerate}

\subsection{Markov-Similar and $V$-Good Pairs}\label{subsection:Markov-similar}
Suppose we are given a closed configuration space $X$, a Markov cocycle $M\in \M_X$ and an interaction $V$ on $X$. If $M$ is not Gibbs with the interaction $V$ we might be still interested in the extent to which it is not. An asymptotic pair $(x,y)\in \Delta_X$ is called \emph{(M,V)-good} if
$$M(x,y)=\sum_{S\subset \V\text{ finite}}\left(V(y|_S)- V(x|_S)\right).$$
In most cases the Markov cocycle $M$ will be fixed, so we will drop $M$ and call a pair $V$-good instead of $(M,V)$-good. An asymptotic pair $(x,y)\in \Delta_X$ is said to be \emph{Markov-similar} to $(x',y')$ if there is a finite set $A\subset \V$ such that
\begin{eqnarray*}
x_u&=&y_u\text{, }\\
x'_u&=&y'_u\text{ for }u\in A^c
\end{eqnarray*}•\
and
\begin{eqnarray*}
x_u&=&x'_u\text{,}\\
y_u&=&y'_u\text{ for }u \in A\cup \partial A.
\end{eqnarray*}
It follows that if $M$ is a Markov cocycle on $X$ and $(x,y), (x',y')\in \Delta_X$ are Markov-similar then $M(x,y)=M(x',y')$. Being $V$-good is infectious.
\begin{prop}\label{prop:V-goodness_equivalence_markov_similar}
Let $X$ be an n.n.constraint space, $M$ a Markov cocycle and $V$ a nearest neighbour interaction on $X$. The set of $V$-good pairs is an equivalence relation on $X$. Additionally if $(x,y),(x',y')\in \Delta_X$ are Markov similar then $(x,y)$ is $V$-good if and only if $(x',y')$ is $V$-good.
\end{prop}
\begin{proof} The reflexivity and symmetry of the relation $V$-good follows from \eqref{equation:flip_gives_negative_same_0} and the cocycle condition implies that the relation is transitive. Thus the relation is an equivalence relation.

Let $(x,y), (x',y')\in \Delta_X$ be Markov-similar pairs. Since $M$ is a Markov cocycle
\begin{equation}
M(x,y)= M(x',y').\label{equation:Markovsimilarcocycleequal}
\end{equation}•
Let $A\subset \V$ be a finite set such that
$$x_u=x'_u\text{ and }y_u=y'_u$$
for $u\in A\cup \partial A$ and
$$x_u=y_u\text{ and }x'_u=y'_u$$
for $u \in A^c$. If $S\subset \V$ is a clique then either $S\subset A\cup \partial A$ or $S\subset A^c$.
If $S\subset A\cup \partial A$ then
$$x|_S= x'|_S\text{ and }y|_S=y'|_S$$
implying
$$V(y|_S)-V(x|_S)=V(y'|_S)-V(x'|_S).$$
If $S\subset A^c$ then
$$x|_S=y|_S\text{ and }x'|_S=y'|_S$$
implying
$$V(y|_S)-V(x|_S)=V(y'|_S)-V(x'|_S)=0.$$
Since $V$ is a nearest neighbour interaction
$$\sum_{S\subset\V\text{ finite}}V(y|_S)-V(x|_S)=\sum_{S\subset\V\text{ finite}}V(y'|_S)-V(x'|_S).$$
Since $(x,y)$ is a $V$-good pair by \eqref{equation:Markovsimilarcocycleequal}
$$M(x',y')=M(x,y)=\sum_{S\subset \V\text{ finite}}\left(V(y|_S)- V(x|_S)\right)=\sum_{S\subset\V\text{ finite}}V(y'|_S)-V(x'|_S)$$
completing the proof.
\end{proof}

\begin{corollary}\label{corollary:chain_of_V-goodness}
Let $X$ be an n.n.constraint space, $M$ a Markov cocycle and $V$ a nearest neighbour interaction on $X$. Suppose for some $(x,y)\in \Delta_X$ there exists a chain $x=x^1, x^2, x^3, \ldots, x^n=y$ such that each $(x^i, x^{i+1})\in \Delta_X$ and is Markov similar to a $V$-good pair. Then $(x,y)$ is $V$-good.
\end{corollary}

This follows from Proposition \ref{prop:V-goodness_equivalence_markov_similar}

\subsection{Graph and Strong Config-Folding}
\label{subsection:Graphfolding}
We shall now introduce graph folding and extract some of its properties so as to define folding for configuration spaces. Graph folding was introduced in \cite{nowakowskiwinkler} and used in \cite{brightwell2000gibbs} so as to prove a slew of properties which are satisfied by a given graph if and only if it is satisfied by its folds. Fix some finite undirected graph $\H=(\V_\H,\E1_\H)$ without multiple edges. For any vertex $a\in \H$ we say that $\H\setminus\{a\}$ is a \emph{fold} of the graph $\H$ if there exists $b\in \H\setminus\{a\}$ such that
$$\{c\:|\: c\sim_\H a\}\subset \{c\:|\: c\sim_\H b\}.$$
In such a case we say that $a$ is folded into $b$.

For example in the 4-cycle $C_4$ the vertex $3$ can be folded into the vertex $1$. However no vertex can be folded in the 3-cycle $C_3$.
\begin{figure}[h]
\includegraphics[angle=0,
width=.3\textwidth]{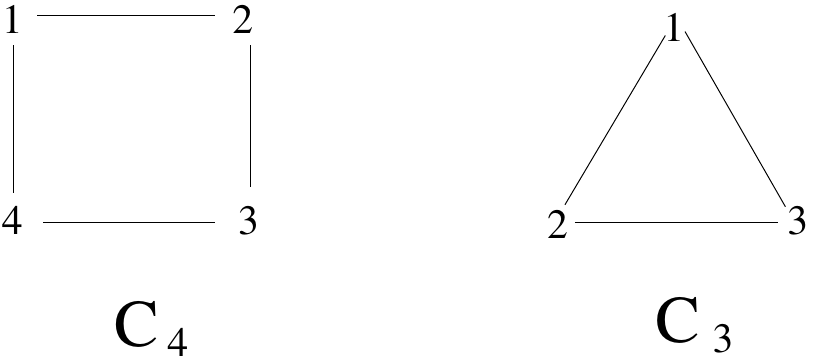}
\caption{$C_4$ and $C_3$}
\end{figure}

Given a graph $\G=(\V, \E1)$ and $v\in \V$, the \emph{$n$-ball around $v$} is given by
$$D_n(v):=\{w\in \V\:|\: d_\G(v,w)\leq n\}$$
where $d_\G$ is the graph distance on $\G$.
Given a symbol $a$ and $F\subset \V$ we denote by $a^F$ the pattern on $F$ given by $a^F_v=a$ for all $v\in F$.

We wish to generalise the following property:

\begin{prop} Consider a bipartite graph $\G=(\V,\E1)$, a graph $\H=(\V_\H, \E1_\H)$ and vertices $a,b\in \V_\H$ where the vertex $a$ can be folded into the vertex $b$. Let $X=Hom(\G, \H)$. Then for all edges $(v_1,v_2), (v_2, v_3)\in \E1$ and $c\in \V_\H$, $\l a,c\r_{\{v_1, v_2\}}\in \B_{\{v_1, v_2\}}(X)$ implies
\begin{eqnarray*}
\l b,c\r_{\{v_1, v_2\}}\in \B_{\{v_1, v_2\}}(X)\\
\l c,b\r_{\{v_2, v_3\}}\in \B_{\{v_2, v_3\}}(X)\text{ and}\\
b^{\partial D_1(v_1)}\in \B_{\partial D_1(v_1)}(X).
\end{eqnarray*}•
\end{prop}
\begin{proof} Since $a\sim_\H c$ and $a$ can be folded into the vertex $b$ we have $b\sim_\H c$. Consider partite classes $P_1, P_2\subset \V$ of $\G$ such that $v_1\in P_1$. Then the configuration $x\in \V^{\V_\H}$ given by
$$x_{v}:=\begin{cases}
b\text{ if }v\in P_1\\
c\text{ if }v\in P_2\end{cases}$$
is an element of $Hom(\G, \H)$. Thus
\begin{gather*}
\l b,c\r_{\{v_1, v_2\}}=x|_{\{v_1,v_2\}}\in \B_{\{v_1, v_2\}}(X)\\
\l c,b\r_{\{v_2, v_3\}}=x|_{\{v_2,v_3\}}\in \B_{\{v_2, v_3\}}(X)\text{ and}\\
b^{\partial D_1(v_1)}=x|_{\partial D_1(v_1)}\in \B_{\partial D_1(v_1)}(X).
\end{gather*}•
\end{proof}

For the rest of the paper fix a bipartite graph $\G=(\V, \E1)$. Let $X\subset \A^\V$ be an n.n.constraint space. Given distinct symbols $a,b\in \A$, we say that $a$ can be \emph{strongly config-folded} into $b$ if for all edges $(v_1,v_2), (v_2, v_3)\in \E1$ and $c\in \A$, $\l a,c\r_{\{v_1, v_2\}}\in \B_{\{v_1, v_2\}}(X)$ implies
\begin{eqnarray}
\l b,c\r_{\{v_1, v_2\}}\in \B_{\{v_1, v_2\}}(X), \label{equation:foldingdefn1}\\
\l c,b\r_{\{v_2, v_3\}}\in \B_{\{v_2, v_3\}}(X)\text{ and}\label{equation:foldingdefn2}\\
b^{\partial D_1(v_1)}\in \B_{\partial D_1(v_1)}(X).\label{equation:foldingdefn3}
\end{eqnarray}•
In such a case, $X\cap (\A\setminus\{a\})^\V$ is called a \emph{strong config-fold} of $X$ and $X$ is called a \emph{strong config-unfold} of $X\cap (\A\setminus\{a\})^\V$. Note that $X\cap (\A\setminus\{a\})^\V$ is still an n.n.constraint space and is obtained by forbidding the symbol $a$ in $X$. Further if $X$ is invariant under a subgroup $Gr\subset Aut(\G)$ then $X\cap (\A\setminus\{a\})^\V$ is also invariant under $Gr$. Let $X_{a}$ denote the strong config-fold $X\cap (\A\setminus\{a\})^\V$. The idea of folding is captured by \eqref{equation:foldingdefn1} while \eqref{equation:foldingdefn2} and \eqref{equation:foldingdefn3} are reminiscent of homomorphism spaces. Indeed if an n.n.constraint space $X$ satisfies \eqref{equation:foldingdefn1} then for all $x\in X$ and $v\in \V$ such that $x_v=a$, the configuration $\theta^v_b(x)\in X$. Thus if $a$ strongly config-folds into $b$ then any appearance of $a$ in any configuration in $X$ can be replaced by $b$. Recall that a safe symbol can replace any other symbol. Thus the notion of strong config-folding generalises the notion of a safe symbol. This notion is stronger than the notion of config-folding as introduced in \cite[Section 5]{chandgotiafourcycle2015}.

\begin{prop}\label{prop:fold_generalise_fold}
Let $\G=(\V, \E1)$ be a bipartite graph and $\A$ be a finite set. Let $X\subset \A^\V$ be an n.n.constraint space with a safe symbol $\star$. Then any symbol $a\in \A\setminus\{\star\}$ can be strongly config-folded into $\star$. The resulting strong config-fold $X_a$ is also an n.n.constraint space with the same safe symbol $\star$.
\end{prop}

Indeed $X_a$ is obtained just by forbidding the symbol $a$ from $X$ and $\star$ is still a safe symbol. In general it is not necessary that the symbol being strongly config-folded into has to be a safe symbol. For instance given any bipartite graph $\G$ the space $Hom(\G,C_4)$, can be strongly config-folded in two steps to $Hom(\G,Edge)$, yet $C_4$ does not have any safe symbol. Note that the strong config-unfold of an n.n.constraint space with a safe symbol need not have a safe symbol.
For example if $\H$ is the graph given by Figure \ref{figure:safe_symbol}
\begin{figure}[h]
\includegraphics[angle=0,
width=.1\textwidth]{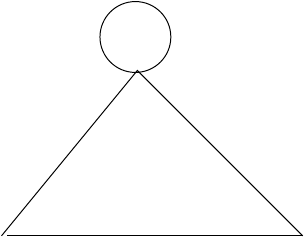}
\caption{Graph for a Space of Homomorphisms with a Safe Symbol\label{figure:safe_symbol}}
\end{figure}
then for any bipartite graph $\G$ the top vertex is a safe symbol in the space $Hom(\G, H)$.
\noindent However if we attach trees to $\H$ to obtain $\H^\prime$ given by Figure \ref{figure:fold_to_safe}
\begin{figure}[h]
\includegraphics[angle=0,
width=.2\textwidth]{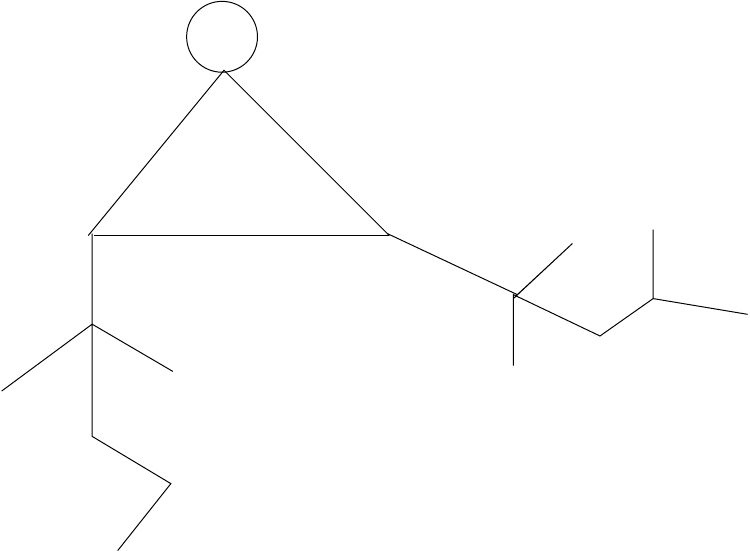}
\caption{Dismantlable Graph for a Space of Homomorphisms without a Safe Symbol\label{figure:fold_to_safe}}
\end{figure}
then $Hom(\G, \H^\prime)$ does not have any safe symbol but can be strongly config-folded into $Hom(\G, \H)$ by folding in the trees attached to $\H$.

Strong config-folding induces a natural map between the spaces of configurations and their cocycles as demonstrated by the following proposition.

\begin{prop} \label{proposition:induced_map_cocycles}
Let $\G=(\V,\E1)$ be a bipartite graph, $\A$ be a finite set and $Gr\subset Aut(\G)$ be a subgroup. Let $X\subset \A^\V$ be a $Gr$-invariant n.n.constraint space on $\G$ and $X_a$ be its strong config-fold for some $a\in \A$. Then the linear map $F: \M^{Gr}_X\longrightarrow \M^{Gr}_{X_a}$ given by $F(M):= M|_{\Delta_{X_a}}$ is surjective and $F(\Gi^{Gr}_X)=\Gi^{Gr}_{X_a}$.
\end{prop}

\begin{proof}

Let $a$ strongly config-fold into $b\in \A$. If $M\in \Gi^{Gr}_X$ then the restriction of the $Gr$-invariant nearest neighbour interaction for $M$ to $X_a$ gives us a $Gr$-invariant nearest neighbour interaction for $F(M)$ proving that $F(M)\in \Gi^{Gr}_{X_a}$. Thus $F(\Gi^{Gr}_X)\subset\Gi^{Gr}_{X_a}$. We will construct a map $\phi^\star: \M^{Gr}_{X_a}\longrightarrow \M^{Gr}_X$ such that $\phi^\star(\Gi^{Gr}_{X_a})\subset \Gi^{Gr}_{X}$ and $F\circ \phi^\star$ is the identity map on $\M^{Gr}_{X_a}$. Note that this is sufficient to conclude that $F$ is surjective and $F(\Gi^{Gr}_X)=\Gi^{Gr}_{X_a}$ thereby completing the proof.

The strong config-folding induces a mapping $\phi:X\longrightarrow X_a$ given by
\begin{equation*}
\phi(x)_v:=\begin{cases}x_v&\text{ if }x_v\neq a\\
b&\text{ if }x_v =a
\end{cases}
\end{equation*}•
for all $x\in X$ and $v\in \V$. Let $g\in Gr$ and $x\in X$. Then
$$(\phi(gx))_v=\begin{cases}
(gx)_v=x_{g^{-1}v}&\text{ if }x_{g^{-1}v}\neq a\\
b&\text{ if }(gx)_v=x_{g^{-1}v}=a
\end{cases}•$$
and
$$(g(\phi(x))_v=(\phi(x))_{g^{-1}v}=\begin{cases}
x_{g^{-1}v}&\text{ if }x_{g^{-1}v}\neq a\\
b&\text{ if }x_{g^{-1}v} =a.
\end{cases}•$$
Therefore $\phi$ commutes with the action of $Gr$. Note that $\phi|_{X_a}$ is the identity.

The map $\phi$ in turn induces a map between the cocycles which we shall now describe. Let $M \in \M_{X_a}^{Gr}$ be a Markov cocycle. Consider $M^\prime: \Delta_X\longrightarrow \R$ given by
$$M^\prime(x,y):= M(\phi(x), \phi(y)).$$
We will prove that $M^\prime\in \M^{Gr}_X$. 

\noindent\emph{Cocycle condition:} If $(x,y), (y,z)\in \Delta_X$ then
$$M^\prime(x,y)+ M^\prime(y,z)=M(\phi(x), \phi(y))+M(\phi(y), \phi(z))= M(\phi(x),\phi(z))= M^\prime(x,z).$$

\noindent\emph{Markov condition: } If $(x,y), (x',y')\in \Delta_X$ are Markov-similar then $(\phi(x), \phi(y)), (\phi(x'), \phi(y'))\in \Delta_{X_{a}}$ are Markov-similar as well implying $M(\phi(x),\phi(y))= M(\phi(x'), \phi(y'))$ and thus 
\begin{eqnarray*}
M^\prime(x, y)&=&M(\phi(x), \phi(y))\\
&=& M(\phi(x'), \phi(y'))\\
&=& M^\prime(x',y')
\end{eqnarray*}•
which verifies the Markov condition for $M^\prime$.

\noindent\emph{$Gr$-invariance condition: }
Since $\phi$ commutes with the action of $Gr$, for all $g\in Gr$
$$M^\prime(gx,gy)=M(\phi(gx), \phi(gy))=M(g(\phi(x)), g(\phi(y)))=M(\phi(x), \phi(y))= M^\prime(x,y).$$
Hence $M^\prime\in \M^{Gr}_X$. Moreover if $M\in \Gi^{Gr}_{X_a}$ with a $Gr$-invariant nearest neighbour interaction $V$, then for all $(x,y)\in \Delta_{X}$
$$M^\prime(x,y)=M(\phi(x),\phi(y))= \sum_{A\subset \V\text{ finite }}V([\phi(y)]_A)-V([\phi(x)]_A)$$
proving that $V\circ \phi$ is a $Gr$-invariant nearest neighbour interaction for $M^\prime$.

Thus the map $\phi^\star: \M^{Gr}_{X_a}\longrightarrow \M^{Gr}_{X}$ given by
$$\phi^\star(M)(x,y):=M(\phi(x), \phi(y))$$
satisfies $\phi^\star(\Gi^{Gr}_{X_a})\subset \Gi^{Gr}_{X}$. Moreover since $\phi|_{X_a}$ is the identity map on $X_a$ therefore $\phi^\star( M)|_{\Delta_{X_a}}=M$ for all $M\in \M^{Gr}_{X_a}$ proving $F\circ \phi^\star$ is the identity map on $\M_{X_a}$.
\end{proof}

Given a $Gr$-invariant topological Markov field $Y \subset X$ there is always a linear map $F:\M^{Gr}_X\longrightarrow \M^{Gr}_Y$ given by $F(M)=M|_{\Delta_{Y}}$ and $F(\Gi^{Gr}_X)\subset \Gi^{Gr}_Y$. However if $Y$ cannot be obtained by a sequence of strong config-folds starting with $X$, then this map need not be surjective. Indeed, consider the following example:

Let $\H$ be the graph given by Figure \ref{figure:safe_symbol} and fix $d\geq 2$. Let $X=Hom(\Z^d, \H)$ and $Y=Hom(\Z^d,C_3)$. Since there is a graph embedding from the 3-cycle $C_3$ to $\H$ it follows that $Hom(\Z^d,C_3)\subset Hom(\Z^d,\H)$. Let $\sigma$ denote the group of translations of the $\Z^d$ lattice. Since the top vertex of $\H$ is a safe symbol for $Hom(\Z^d,\H)$ it follows from the strong Hammersley-Clifford theorem (Theorem \ref{Theorem:Strong_Hammersley-Clifford}) that $\M^\sigma_X=\Gi^\sigma_X$. Therefore $F(\M^\sigma_X)\subset \Gi^\sigma_{Y}$. However by Proposition 5.3 in \cite{chandgotiameyerovitch}, $\Gi^\sigma_Y\subsetneq\M^\sigma_Y$. It follows that $F(\M^\sigma_X)\subsetneq \M^\sigma_{Y}$.

\section{The Main Results}\label{section: main results}
\begin{thm}\label{thm:main theorem}
Let $\G=(\V, \E1)$ be a bipartite graph and $X$ be a Hammersley-Clifford n.n.constraint space on $\G$. Then the strong config-folds and strong config-unfolds of $X$ are also Hammersley-Clifford.
\end{thm}

The $Gr$-invariant version of Theorem \ref{thm:main theorem} holds as well.
\begin{thm}\label{thm:G-invariant main theorem}
Let $\G=(\V, \E1)$ be a bipartite graph, $Gr\subset Aut(\G)$ be a subgroup and $X$ be a $Gr$-Hammersley-Clifford n.n.constraint space on $\G$. Then the strong config-folds and strong config-unfolds of $X$ are also $Gr$-Hammersley-Clifford.
\end{thm}

We know that all frozen spaces of configurations are $Gr$-Hammersley-Clifford for all subgroups $Gr\subset Aut(\G)$. We can construct many more examples of Hammersley-Clifford spaces by using these theorems.

\begin{enumerate}

\item
N.N.Constraint space with a safe symbol and dismantlable graphs.

By Proposition \ref{prop:fold_generalise_fold} starting with an n.n.constraint space with a safe symbol $\star$ we can strongly config-fold all the symbols one by one into the symbol $\star$ resulting in $\{\star\}^\V$ which is frozen. Thus these theorems generalise Theorem \ref{Theorem:Strong_Hammersley-Clifford} in the case when $\G$ is a bipartite graph. Furthermore any configuration space which can be strongly config-folded into a space with a safe symbol is still Hammersley-Clifford. For instance given the graph $\H^\prime$ in Figure \ref{figure:fold_to_safe}, even though $Hom(\G, \H^\prime)$ does not have any safe symbol, it is $Gr$-Hammersley-Clifford for any subgroup $Gr\subset Aut(\G)$.

More generally a graph $\H$ is called \emph{dismantlable} if there exists a sequence of folds on the graph leading to a single vertex (with or without a self-loop). By these theorems, if $\H$ is dismantlable then $Hom(\G, \H)$ is $Gr$-Hammersley-Clifford for any subgroup $Gr\subset Aut(\G)$. This provides a large class of examples. Consider for instance the space $Hom(\G, \H_{n,m})$ where $\H_{n,m}$ is a graph with vertices $\V_{\H_{n,m}}:=\{1, 2, \ldots, n\}$ and edges given by $(i,j)\in \E1_{\H_{n,m}}$ if and only if $|i-j|\leq m$.\\
\begin{figure}[h]
\includegraphics[angle=0,
width=.4\textwidth]{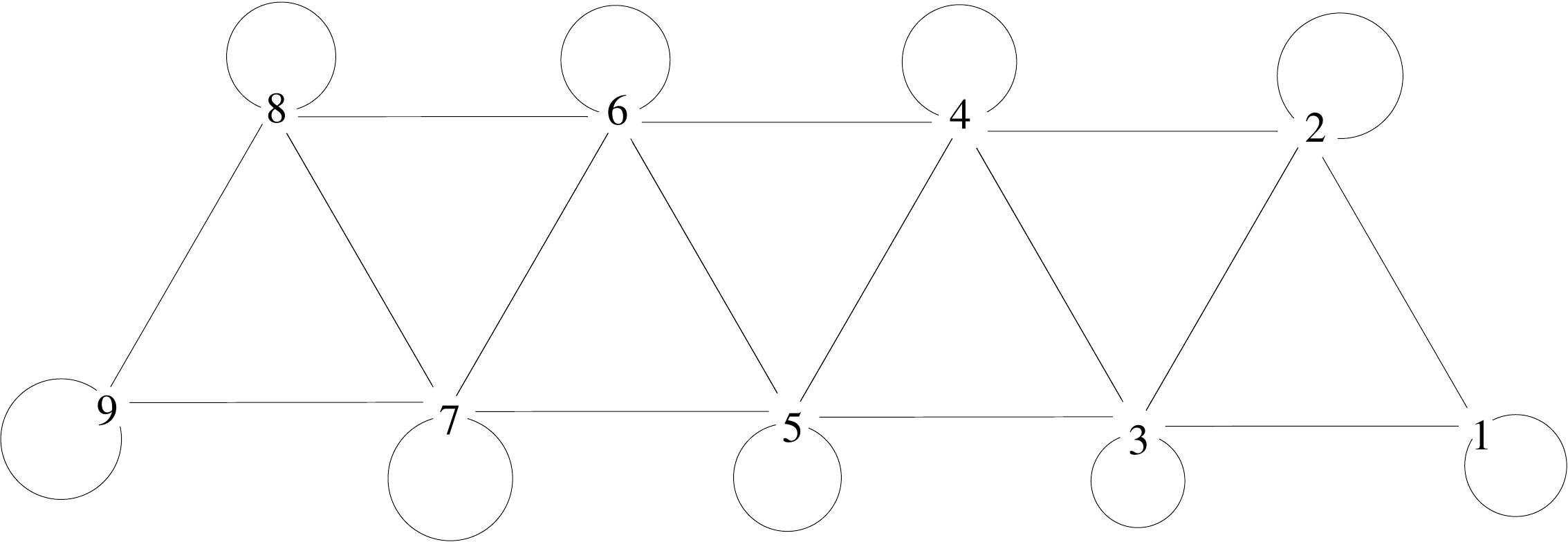}
\caption{$\H_{9,2}$}
\end{figure}

The sequence of folds $1$ to $2$, $2$ to $3$, $3$ to $4$, \ldots, $n-1$ to $n$ yields the space $\{n\}^{\G}$ from $Hom(\G, \H_{n,m})$ proving that it is $Gr$-Hammersley-Clifford for any subgroup $Gr\subset Aut(\G)$.

\item
$Hom(\G,Edge)$ where $Edge$ consists of two vertices $0$ and $1$ connected by a single edge.

By these theorems a configuration space which can be strongly config-folded into $Hom(\G,Edge)$ is still Hammersley-Clifford. For example if $\H$ is the graph given by Figure \ref{figure:foldtoedge}
\begin{figure}[h]
\includegraphics[angle=0,
width=.2\textwidth]{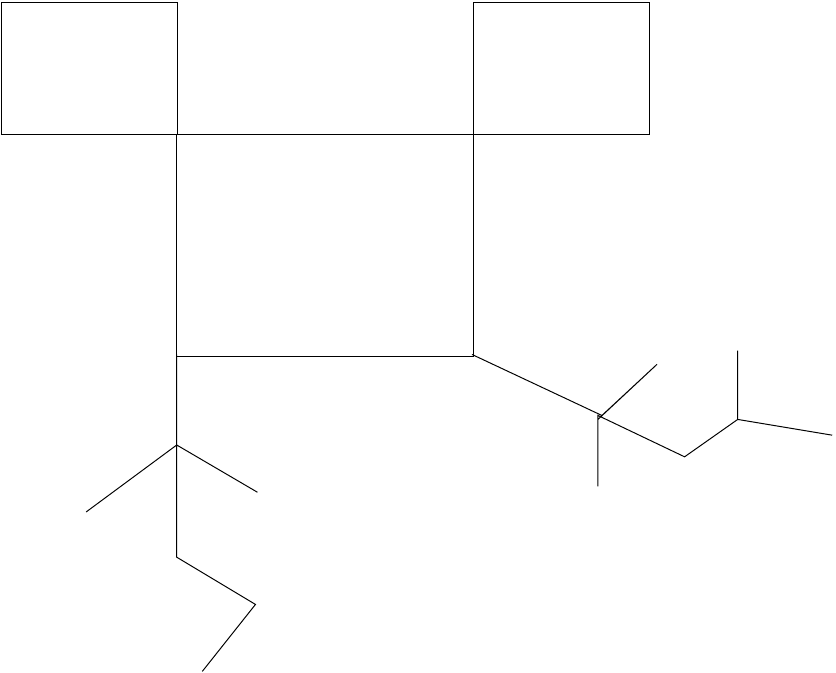}
\caption{A Graph which Folds to a Single Edge \label{figure:foldtoedge}}
\end{figure}
then it can be folded to the graph isomorphic to $Edge$ and hence $Hom(\G, \H)$ is $Gr$-Hammersley-Clifford for any subgroup $Gr\subset Aut(\G)$.
\end{enumerate}

Note that although these are homomorphism spaces, the theorems are true in the general setting of configuration spaces. These specific examples have been chosen for convenience.

\subsection{A Concrete Example:}\label{subsection:concrete}
We will first work out the following example to illustrate the key ideas of the proof. 

Suppose $\H$ and $\H^\prime$ are graphs given by Figure \ref{figure:Hprime}.
\begin{figure}[h]
\includegraphics[angle=0,
width=.4\textwidth]{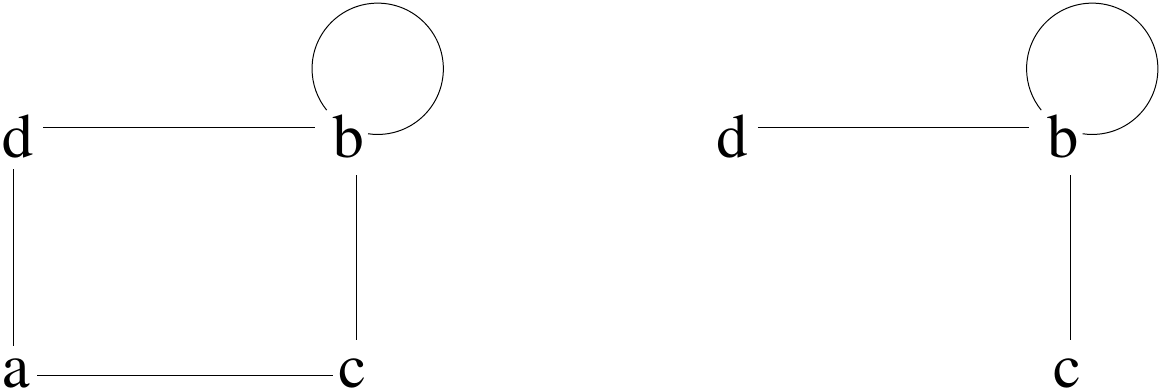}
\caption{Graphs $\H$ and $\H^\prime$\label{figure:Hprime}}
\end{figure}
Let $X= Hom(\Z^2, \H)$. Then by folding the vertex $a$ into the vertex $b$ we obtain the space $X_a=Hom(\Z^2, \H^\prime)$.

Note that $X$ does not have any safe symbol but $b$ is a safe symbol for $X_a$. Let $\sigma\subset Aut(\Z^2)$ denote the subgroup of all translations of $\Z^2$. By the strong Hammersley-Clifford theorem (Theorem \ref{Theorem:Strong_Hammersley-Clifford}) $X_a$ is $\sigma$-Hammersley-Clifford. We will prove that $X$ is $\sigma$-Hammersley-Clifford.

Let $M\in \M^\sigma_X$ be a $\sigma$-invariant Gibbs cocycle. Then $M|_{\Delta_{X_a}}$ is a $\sigma$-invariant Markov cocycle on $X_a$ and hence a Gibbs cocycle with some $\sigma$-invariant nearest neighbour interaction, which we will call $V$.

Given $e, f,g, h, i\in \V_\H$ and $v\in \Z^2$ let $ \left\l\begin{smallmatrix}&e&\\f&g&h\\&i&\end{smallmatrix}\right\r^v$ denote the configuration
$$\left\l\begin{smallmatrix}&e&\\f&g&h\\&i&\end{smallmatrix}\right\r^v:=\begin{cases}
g&\text{ if }u=v\\
e&\text{ if }u=v+(0,1)\\
f&\text{ if }u=v-(1,0)\\
h&\text{ if }u=v+(1,0)\\
i&\text{ if }u=v-(0,1)\\
b&\text{ if }u\in D_1(v)^c.
\end{cases}$$
For all $v\in \Z^2$ consider $x^v:= \left\l\begin{smallmatrix}&d&\\d&a&d\\&d&\end{smallmatrix}\right\r^v \in X_\H$. Consider a $\sigma$-invariant nearest neighbour interaction $V^\prime$ as follows:
\begin{enumerate}
\item If $v \sim_{\Z^2} w \in \Z^2$, $\l e,f\r_{\{v,w\}}\in \B_{\{v,w\}}(X_a)$ then
\begin{eqnarray}
V^\prime(\l e,f\r_{\{v,w\}})&=&V(\l e,f\r_{\{v,w\}})\text{ and}\nonumber\\
V^\prime(\l e\r_v)&=&V(\l e\r_v).\label{equation:1_concrete_interaction}
\end{eqnarray}•
\item
The interaction between $a$ and $d$ is $0$, that is, for all $v \sim_{\Z^2} w \in \Z^2$
\begin{eqnarray}
V^\prime(\l a,d\r_{\{v,w\}})=0.\label{equation:2_concrete_interaction}
\end{eqnarray}•
\item
The single site interaction for $\l a\r_v$ for all $v\in\Z^2$ is given by
\begin{eqnarray*}
V^\prime(\l a\r_v)&=&M\left( \left\l \begin{smallmatrix}&d&\\d&b&d\\&d&\end{smallmatrix}\right\r^v, \left\l\begin{smallmatrix}&d&\\d&a&d\\&d&\end{smallmatrix}\right\r^v\right)+V(\l b\r_v)+V(\l b,d\r_{\{v,v+(1,0)\}})+V(\l b,d\r_{\{v,v-(1,0)\}})\\
&&+V(\l b,d\r_{\{v,v+(0,1)\}})+V(\l b,d\r_{\{v,v-(0,1)\}}).
\end{eqnarray*}
By \eqref{equation:1_concrete_interaction} and \eqref{equation:2_concrete_interaction} this implies that the pair $ \left(\left\l\begin{smallmatrix}&d&\\d&b&d\\&d&\end{smallmatrix}\right\r^v, \left\l\begin{smallmatrix}&d&\\d&a&d\\&d&\end{smallmatrix}\right\r^v\right)$ is $V^\prime$-good.
\item
Let
\begin{eqnarray*}
V^\prime(\l a,c\r_{\{v, v+(1,0)\}})&=&M\left( \left\l\begin{smallmatrix}&d&\\d&a&d\\&d&\end{smallmatrix}\right\r^v, \left\l\begin{smallmatrix}&d&\\d&a&c\\&d&\end{smallmatrix}\right\r^v\right)+V(\l d\r_{v+(1,0)})-V(\l c\r_{v+(1,0)})\\
&&+V(\l d,b\r_{\{v+(1,0), v+(1,1)\}})+V(\l d,b\r_{\{v+(1,0), v+(2,0)\}})\\
&&+V(\l d,b\r_{\{v+(1,0), v+(1,-1)\}})-V(\l c,b\r_{\{v+(1,0), v+(1,1)\}})\\
&&-V(\l c,b\r_{\{v+(1,0), v+(2,0)\}})-V(\l c,b\r_{\{v+(1,0), v+(1,-1)\}}).
\end{eqnarray*}•
By \eqref{equation:1_concrete_interaction} and \eqref{equation:2_concrete_interaction} the previous equation implies that the pair $ \left(\left\l\begin{smallmatrix}&d&\\d&a&d\\&d&\end{smallmatrix}\right\r^v, \left\l\begin{smallmatrix}&d&\\d&a&c\\&d&\end{smallmatrix}\right\r^v\right)$ is $V^\prime$-good. Similarly we can define $V^\prime(\l a,c\r_{\{v,v-(1,0)\}})$, $V^\prime(\l a,c\r_{\{v,v+(0,1)\}})$ and $V^\prime(\l a,c\r_{\{v,v-(0,1)\}})$, the corresponding expressions of which will imply that the pairs $ \left(\left\l\begin{smallmatrix}&d&\\d&a&d\\&d&\end{smallmatrix}\right\r^v, \left[\begin{smallmatrix}&d&\\c&a&d\\&d&\end{smallmatrix}\right]^v\right)$, \\$ \left(\left\l\begin{smallmatrix}&d&\\d&a&d\\&d&\end{smallmatrix}\right\r^v, \left\l\begin{smallmatrix}&c&\\d&a&d\\&d&\end{smallmatrix}\right\r^v\right)$, $ \left(\left\l\begin{smallmatrix}&d&\\d&a&d\\&d&\end{smallmatrix}\right\r^v, \left\l\begin{smallmatrix}&d&\\d&a&d\\&c&\end{smallmatrix}\right\r^v\right)$
are $V^\prime$-good.
\end{enumerate}•

Since $V$ and $M$ are $\sigma$-invariant it follows that $V^\prime$ is also $\sigma$-invariant. We want to prove that $V^\prime$ is an interaction for $M$. Equivalently we want to prove that all asymptotic pairs are $V^\prime$-good. Let $(x,y)\in \Delta_X$. Since any appearance of $a$ in the configurations in $X$ can be replaced by $b$, by replacing all the $a$'s outside the set of sites where $x$ and $y$ differ and its boundary we can obtain a pair $(x^1,y^1)\in \Delta_{X}$ which is Markov-similar to $(x,y)$ and has finitely many $a$'s. Thus by Proposition \ref{prop:V-goodness_equivalence_markov_similar} it is sufficient to prove that pairs $(x,y)\in \Delta_X$ with finitely many $a$'s are $V^\prime$-good. Since the $a$'s can be replaced by $b$'s one by one and any pair in $\Delta_{X_a}$ is $V^\prime$-good by Corollary \ref{corollary:chain_of_V-goodness} it is sufficient to prove that pairs in $X$ in which a single $a$ is replaced by $b$ are $V^\prime$-good. Since $a$ can be folded into $b$ and $\partial\{a\}=\{c,d\}$ any such pair is Markov-similar to a pair $\left(\left\l\begin{smallmatrix}&e&\\f&a&h\\&i&\end{smallmatrix}\right\r^v,\left\l\begin{smallmatrix}&e&\\f&b&h\\&i&\end{smallmatrix}\right\r^v\right)$ for some $v\in \Z^2$ and $e, f,g, h, i\in \{c,d\}$.

The pairs
$$\left(\left\l\begin{smallmatrix}&e&\\f&a&h\\&i&\end{smallmatrix}\right\r^v,\left\l\begin{smallmatrix}&d&\\f&a&h\\&i&\end{smallmatrix}\right\r^v\right), \left(\left\l\begin{smallmatrix}&d&\\f&a&h\\&i&\end{smallmatrix}\right\r^v,\left\l\begin{smallmatrix}&d&\\d&a&h\\&i&\end{smallmatrix}\right\r^v\right),\left(\left\l\begin{smallmatrix}&d&\\d&a&h\\&i&\end{smallmatrix}\right\r^v,\left\l\begin{smallmatrix}&d&\\d&a&d\\&i&\end{smallmatrix}\right\r^v\right),\left(\left\l\begin{smallmatrix}&d&\\d&a&d\\&i&\end{smallmatrix}\right\r^v,\left\l\begin{smallmatrix}&d&\\d&a&d\\&d&\end{smallmatrix}\right\r^v\right)$$
are Markov-similar to
$$\left(\left\l\begin{smallmatrix}&e&\\d&a&d\\&d&\end{smallmatrix}\right\r^v,\left\l\begin{smallmatrix}&d&\\d&a&d\\&d&\end{smallmatrix}\right\r^v\right), \left(\left\l\begin{smallmatrix}&d&\\f&a&d\\&d&\end{smallmatrix}\right\r^v,\left\l\begin{smallmatrix}&d&\\d&a&d\\&d&\end{smallmatrix}\right\r^v\right),\left(\left\l\begin{smallmatrix}&d&\\d&a&h\\&d&\end{smallmatrix}\right\r^v,\left\l\begin{smallmatrix}&d&\\d&a&d\\&d&\end{smallmatrix}\right\r^v\right),\left(\left\l\begin{smallmatrix}&d&\\d&a&d\\&i&\end{smallmatrix}\right\r^v,\left\l\begin{smallmatrix}&d&\\d&a&d\\&d&\end{smallmatrix}\right\r^v\right)$$ 
respectively. Since $e, f,g, h, i\in \{c,d\}$, these pairs are $V^\prime$-good. Thus each adjacent pair in the chain
$$\left\l\begin{smallmatrix}&e&\\f&a&h\\&i&\end{smallmatrix}\right\r^v,\left\l\begin{smallmatrix}&d&\\f&a&h\\&i&\end{smallmatrix}\right\r^v,\left\l\begin{smallmatrix}&d&\\d&a&h\\&i&\end{smallmatrix}\right\r^v,\left\l\begin{smallmatrix}&d&\\d&a&d\\&i&\end{smallmatrix}\right\r^v,\left\l\begin{smallmatrix}&d&\\d&a&d\\&d&\end{smallmatrix}\right\r^v,
\left\l\begin{smallmatrix}&d&\\d&b&d\\&d&\end{smallmatrix}\right\r^v,
\left\l\begin{smallmatrix}&e&\\f&b&h\\&i&\end{smallmatrix}\right\r^v$$
is $V^\prime$-good. By Corollary \ref{corollary:chain_of_V-goodness} the pair $\left(\left\l\begin{smallmatrix}&e&\\f&a&h\\&i&\end{smallmatrix}\right\r^v,\left\l\begin{smallmatrix}&e&\\f&b&h\\&i&\end{smallmatrix}\right\r^v\right)$ is $V^\prime$-good. This completes the proof.

\subsection{Proof of Theorems \ref{thm:main theorem} and \ref{thm:G-invariant main theorem}}

We will now prove Theorems \ref{thm:main theorem} and \ref{thm:G-invariant main theorem}. The proof will give an explicit way of computing the interaction as well. It should also be noted that Theorem \ref{thm:main theorem} is a special case of Theorem \ref{thm:G-invariant main theorem}. Yet we separate the proofs for readability. In the rest of the paper we will denote the adjacency relation in the graph $\G$ by $\sim$ instead of $\sim_\G$. 

\

\noindent\emph{Proof of Theorem \ref{thm:main theorem}.} The bulk of the proof lies in showing that the strong config-unfolds of Hammersley-Clifford spaces are Hammersley-Clifford. We will first prove that the strong config-folds of a Hammersley-Clifford space are Hammersley-Clifford. Let $X\subset \A^\V$ be Hammersley-Clifford and $X_a$ be its strong config-fold. Using Proposition \ref{proposition:induced_map_cocycles} in the case where $Gr=\{id|_\G\}$ we obtain a surjective map $F:\M_X\longrightarrow \M_{X_a}$ such that $F(\Gi_X)=\Gi_{X_a}$. Since $X$ is Hammersley-Clifford, $\M_X=\Gi_X$. Hence
$$\M_{X_a}=F(\M_X)=F(\Gi_X)=\Gi_{X_a}$$
proving that $X_a$ is Hammersley-Clifford.

Now we will prove that strong config-unfolds of Hammersley-Clifford spaces are Hammersley-Clifford spaces as well. Let $X\subset \A^\V$ be an n.n.constraint space and $X_a $ be a strong config-fold of $X$ where $a$ is strongly config-folded into $b$. Let the set of nearest neighbour constraints of $X$ be given by the set $\F_X$. Suppose $X_a$ is Hammersley-Clifford.

Let $M\in \M_X$ be a Markov cocycle. Since $X_a$ is Hammersley-Clifford $M|_{\Delta_{X_a}}\in \Gi_{X_a}$. Let $V$ be a corresponding nearest neighbour interaction. We shall now construct a nearest neighbour interaction $V^\prime$ for $M$. The idea is the following:

Since we have a nearest neighbour interaction for $M|_{\Delta_{X_a}}$ we will change asymptotic pairs in $X$ to asymptotic pairs in $X_a$ using the fewest possible distinct single site changes. These distinct single site changes will correspond to patterns on edges and vertices helping us build $V^\prime$. If we use the single site changes which involve blindly changing the $a$'s into $b$'s we will incur a large number of such changes; hence we have to choose these changes carefully.

\begin{lemma}[Construction of special configurations]\label{lemma:four special configurations}
Let $\G=(\V, \E1)$ be a bipartite graph, $\A$ be a finite set, $X\subset \A^\V$ be an n.n.constraint space on $\G$ and $X_a$ be a strong config-fold of $X$ where the symbol $a$ is strongly config-folded into the symbol $b$. Let
$$\V_1:=\left\{v\in \V\:|\: \text{there exists } w\sim v\text{ such that }\l a,a\r_{\{v,w\}}\in \B_{\{v,w\}}(X)\right\}$$
and
$$\V_2:=\left\{v\in \V\setminus \V_1\:|\: \l a\r_{v}\in \B_{\{v\}}(X)\right\}.$$
For all $v\in \V_1\cup\V_2$ there exists $x^v\in X$ such that
\begin{enumerate}
\item\label{enumerate:configuration1}
If $v\in \V_1$ then $x^v_v=a$ and $x^v|_{D_2(v)\setminus \{v\}}=b^{D_2(v)\setminus \{v\}}$.
\item\label{enumerate:configuration2}
If $v\in \V_2$ then $x^v_{v}=a$ and $x^v|_{\partial D_1(v)}=b^{\partial D_1(v)}$.
\end{enumerate}
Moreover $\theta^v_b(x^v)\in X_a$ and if $w_1, w_2, w_3,\ldots w_r\sim v$ and $c_1, c_2,\ldots,c_r\in\A$ such that $\l a,c_i\r_{\{v,w_i\}}\in \B_{\{v,w_i\}}(X)$ then $\theta^{w_1, w_2, \ldots, w_r}_{c_1, c_2,\ldots, c_r}(x^v)\in X$.
\end{lemma}
\begin{proof}
Let $v\in \V_1$. By \eqref{equation:foldingdefn2} $\l a,b\r_{\{v,w\}}\in \B_{\{v,w\}}(X)$ for all $w\sim v$. Again by \eqref{equation:foldingdefn2} it follows that $\l b,b\r_{\{w, w_1\}}\in \B_{\{w, w_1\}}(X)$ for all $w, w_1\in \V$ such that $w\sim v$ and $w_1\sim w$. Then none of the patterns from $\F_X$, the nearest neighbour constraint set for $X$ appear in $\alpha^v\in\A^{D_2(v)}$ given by
\begin{equation*}
\alpha^v_u:=\begin{cases}
a &\text{ if }u=v\\
b &\text{ if }u \in D_2(v)\setminus\{v\}.
\end{cases}•
\end{equation*}
For $v\in \V_2$ there exists $x^1\in X$ such that $x^1_v=a$. For all $w, w_1\in \V$ such that $w\sim v$ and $w_1\sim w$, \eqref{equation:foldingdefn2} implies that $\l x^1_w,b\r_{\{w,w_1\}}\in \B_{\{w,w_1\}}(X)$. Then none of the patterns from $\F_X$ appear in $\alpha^v\in \A^{D_2(v)}$ given by
\begin{equation*}
\alpha^v_u=\begin{cases}
x^1_u&\text{ if } u\in D_1(v)\\
b&\text{ if } u\in D_2(v)\setminus D_1(v).
\end{cases}•
\end{equation*}•

Fix $v\in \V_1\cup \V_2$. By \eqref{equation:foldingdefn3} there exists $x\in X$ such that $x|_{\partial D_1(v)}=b^{\partial D_1(v)}$. Moreover since $a$ strongly config-folds into $b$ we can assume that $x\in X_a$. Consider $x^v\in \A^\V$ given by
\begin{eqnarray*}
x^v_u:=\begin{cases}
\alpha^v_u \text{ if }u\in D_2(v)\\
x_u\text{ if } u \in D_1(v)^c.
\end{cases}•
\end{eqnarray*}•
The configurations $x^v$ satisfy the Conclusions (\ref{enumerate:configuration1}) and (\ref{enumerate:configuration2}) of this lemma. Since each edge in $\G$ either lies completely in $D_2(v)$ or in $D_1(v)^c$, no subpattern of $x^v$ belongs to $\F_X$. Therefore $x^v\in X$.

Let $v\in \V_1\cup \V_2$. Since $x\in X_a$, $a$ appears in $x^v$ only at $v$ . Moreover since $a$ strongly config-folds into $b$ by \eqref{equation:foldingdefn1}, $\theta^v_b(x^v)\in X_a$. Let $w_1, w_2, w_3,\ldots w_r\sim v$ and $c_1, c_2,\ldots,c_r\in\A$ such that $\l a,c_i\r_{\{v,w_i\}}\in \B_{\{v,w_i\}}(X)$ for all $1\leq i \leq r$. Because the graph is bipartite $w_i\nsim_\G w_j$ for all $1\leq i, j \leq r$. By \eqref{equation:foldingdefn2} for all $w^\prime\sim w_i$ and $1\leq i \leq r$, $\l c_i,b\r_{\{w_i,w^\p\}}\in \B_{\{w_i, w^\p\}}(X)$. By Proposition \ref{prop:changing_remain_nnconstraint} $\theta^{w_1, w_2, \ldots, w_r}_{c_1, c_2, \ldots, c_r}(x^v)\in X$.
\end{proof}

We will now construct an interaction via the following technical lemma.

\begin{lemma}[Construction of $V^\prime$]\label{lemma:construction_of_V^prime}
Let $\G=(\V, \E1)$ be a bipartite graph with partite classes $P_1$ and $P_2$, $\A$ be a finite set, $X\subset \A^\V$ be an n.n.constraint space and $X_a$ be a strong config-fold of $X$ where the symbol $a$ is strongly config-folded into the symbol $b$. Consider sets $\V_1, \V_2\subset \V$ and for all $v\in \V_1\cup \V_2$, configurations $x^v\in X$ satisfying the conclusions of Lemma \ref{lemma:four special configurations}. Let $M\in \M_X$ be a Markov cocycle on $X$ such that $M|_{\Delta_{X_a}}$ is a Gibbs cocycle with interaction $V$. Then there exists a unique nearest neighbour interaction $V^\prime$ on $X$ which satisfies:

\noindent If $v\sim w\in \V$ and $\l c,d\r_{\{v,w\}}\in \B_{\{v,w\}}(X_a)$ then
\begin{eqnarray}
V^\prime(\l c,d\r_{\{v,w\}})&=& V(\l c,d\r_{\{v,w\}})\text{ and}\label{equation:listinteract1}\\
V^\prime(\l c\r_{v})&=&V(\l c\r_{\{v\}}).\label{equation}\label{equation:listinteract2}
\end{eqnarray}
For $v\in \V_1\cup \V_2$ and $w\sim v$
\begin{eqnarray}
V^\prime( x^v|_{\{v,w\}})=0.\label{equation:listinteract3}
\end{eqnarray}
such that the following pairs are $V^\prime$-good:
\begin{enumerate}
\item
$(\tilde{x},\tilde{y}) \in \Delta_{X_a}$.
\label{item:from_X_a}
\item
$(\theta^v_b(x^v), x^v)$ for $v\in \V_1\cup \V_2$.
\label{item:changing_b_at_v}
\item
$ \left(\theta^w_c(x^v),x^v\right)$ for $v\in \V_1\cup \V_2$, $w\sim v$ and $c\in \A\setminus\{a\}$ satisfying $\l a,c\r_{\{v,w\}}\in \B_{\{v,w\}}(X)$.
\label{item:changing_at_boundary_no_a}
\item
$ (\theta^w_a(x^v),x^v)$ for all $v\in \V_1\cap P_1$, $w\sim v$ satisfying $\l a,a\r_{\{v,w\}}\in \B_{\{v,w\}}(X)$.
\label{item:changing_at_boundary_a}
\end{enumerate}
\end{lemma}

In the following proof the reader is encouraged to refer to the statement of Lemma \ref{lemma:four special configurations} for information about configurations $x^v$.
\begin{proof}[Proof of Lemma \ref{lemma:construction_of_V^prime}]
We will begin by proving uniqueness of the interaction assuming its existence. Consider a nearest neighbour interaction $V^\prime$ on $X$ which satisfies the conclusion of this lemma. We will express $V^\prime$ in terms of the cocycle $M$ and the interaction $V$.

Since $V^\prime$ satisfies \eqref{equation:listinteract1}, \eqref{equation:listinteract2} and \eqref{equation:listinteract3} we have to prove that the following can be expressed in terms of $M$ and $V$:
\begin{enumerate}[(a)]
\item For all $v\in \V_1\cup\V_2$, the value $V^\prime(\l a\r_v)$,\label{item:the_single_a_at centre}
\item For all $v\in \V_1\cup\V_2$, $w\sim v$ and $c\in \A\setminus \{x^v_w,a\}$ such that $\l a,c\r_{\{v,w\}}\in \B_{\{v,w\}}(X)$, the value $V^\prime(\l a,c\r_{\{v,w\}})$ and \label{item:a_at_centre_not_a_at_boundary}
\item For all $v\in \V_1\cap P_1$, $w\sim v$ such that $\l a,a\r_{\{v,w\}}\in\B_{\{v,w\}}(X)$, the value $V^\prime(\l a,a\r_{\{v,w\}})$.\label{item:a_at_both}
\end{enumerate}

\noindent\emph{Proof for Part (\ref{item:the_single_a_at centre}):} Let $v\in \V_1\cup\V_2$. Since the pair $(\theta^v_b(x^v),x^v)$ ((\ref{item:changing_b_at_v}) in the statement of the lemma) is $V^\prime$-good by rearranging the expression for $M(\theta^v_b(x^v),x^v)$ we get that
\begin{eqnarray}
V^\prime(\l a\r_v)&=&V^\prime( x^v|_{\{v\}})\nonumber\\
&=&M(\theta^v_b(x^v), x^v)+ V^\prime( \theta^v_b(x^v)|_{\{v\}})+\sum_{w: w\sim v} V^\prime\left( \theta^v_b(x^v)|_{\{v,w\}}\right)\nonumber\\
&&-(\sum_{w: w\sim v} V^\prime\left( x^v|_{\{v,w\}}\right))\label{equation:listsimilargood1}.
\end{eqnarray}
Now we will express the right hand side of this expression in terms of $M$ and $V$. Since $\theta^v_b(x^v)\in X_a$ $V^\prime( \theta^v_b(x^v)|_{\{v,w\}})=V( \theta^v_b(x^v)|_{\{v,w\}})$ and $V^\prime( \theta^v_b(x^v)|_{\{v\}})=V( \theta^v_b(x^v)|_{\{v\}})$. By \eqref{equation:listinteract3}, $V^\prime(x^v|_{\{v,w\}})=0$.

Putting all this together we get
\begin{eqnarray}
V^\prime(\l a\r_v)&=&M(\theta^v_b(x^v), x^v)+ V( \theta^v_b(x^v)|_{\{v\}})+\sum_{w: w\sim v} V( \theta^v_b(x^v)|_{\{v,w\}}).\label{equation:listinteract4}
\end{eqnarray}

\noindent\emph{Proof for Part (\ref{item:a_at_centre_not_a_at_boundary}): }Consider $v\in \V_1\cup \V_2$, $w\sim v$ and $c\in \A\setminus\{a, x^v_w\}$ such that $\l a,c\r_{\{v,w\}}\in \B_{\{v,w\}}(X)$. Since the pair $(\theta^w_c(x^v),x^v)$ ((\ref{item:changing_at_boundary_no_a}) in the statement of the lemma) is $V^\prime$-good, by rearranging the expression for $M(x^v, \theta^w_c(x^v))$ we get
\begin{eqnarray}
V^\prime(\l a,c\r_{\{v,w\}})&=&V^\prime( \theta^w_c(x^v)|_{\{v,w\}})\nonumber\\
&=&M(x^v, \theta^w_c(x^v))+\sum_{w^\prime: w^\prime\sim w}V^\prime( x^v|_{\{w^\prime,w\}}) +V^\prime( x^v|_{\{w\}})\nonumber\\
&&-\left(\sum_{w^\prime: w^\prime\sim w, w^\prime\neq v}V^\prime( \theta^w_c(x^v)|_{\{w^\prime,w\}})\right)-V^\prime( \theta^w_c(x^v)|_{\{w\}})\label{equation:listsimilargood2}.
\end{eqnarray}•
We will now express the right hand side of this expression in terms of $M$ and $V$.

By \eqref{equation:listinteract3}, $V^\prime( x^v|_{\{v,w\}})=0$. We know that $(\theta^w_c(x^v))_{w},\  x^v_w\neq a$ and if $w^\prime\sim w$, $w^\prime\neq v$ then $w^\prime\in \partial D_1(v)$ and so $(\theta^w_c(x^v))_{w^\prime}=x^v_{w^\prime}=b$. Therefore by \eqref{equation:listinteract1} and \eqref{equation:listinteract2} 
$$V^\prime( x^v|_{\{w^\prime,w\}})=V( x^v|_{\{w^\prime,w\}}),\ V^\prime(\theta^w_c(x^v)|_{\{w^\prime, w\}})=V( \theta^w_c(x^v)|_{\{w^\prime, w\}})$$
and
$$V^\prime( x^v|_{\{w\}})=V(x^v|_{\{w\}}),\ V^\prime( \theta^w_c(x^v)|_{\{w\}})=V( \theta^w_c(x^v)|_{\{w\}}).$$

Putting all this together we get
\begin{eqnarray}
V^\prime(\l a,c\r_{\{v,w\}})&=&M(x^v, \theta^w_c(x^v))+\displaystyle{\sum_{w^\prime: w^\prime\sim w, w^\prime\neq v}\left( V( x^v|_{\{w^\prime,w\}})-V(\theta^w_c(x^v)|_{\{w^\prime,w\}}) \right) }\nonumber\\
&&+V( x^v|_{\{w\}})-V( \theta^w_c(x^v)|_{\{w\}}).\label{equation:listinteract5}
\end{eqnarray}•

\noindent\emph{Proof for Part (\ref{item:a_at_both}):} Consider $v\in \V_1\cap P_1$ and $w\sim v$ such that $\l a,a\r_{\{v,w\}}\in \B_{\{v,w\}}(X)$. Since the pair $(\theta^w_a(x^v),x^v)$ ((\ref{item:changing_at_boundary_a}) in the statement of the lemma) is $V^\prime$-good by rearranging the expression for $M(x^v, \theta^w_a(x^v))$ we get that
\begin{eqnarray}
V^\prime(\l a,a\r_{\{v,w\}})&=&V^\prime( \theta^w_a(x^v)|_{\{v,w\}})\nonumber\\
&=&M(x^v, \theta^w_a(x^v))+\sum_{w^\prime: w^\prime\sim w}V^\prime( x^v|_{\{w^\prime,w\}}) +V^\prime( x^v|_{\{w\}})\nonumber\\
&&-\left(\sum_{w^\prime: w^\prime\sim w, w^\prime\neq v}V^\prime( \theta^w_a(x^v)|_{\{w^\prime,w\}})\right)-V^\prime( \theta^w_a(x^v)|_{\{w\}}).\label{equation:listsimilargood3}
\end{eqnarray}•
We will now express the right hand side of this expression in terms of $M$ and $V$. By \eqref{equation:listinteract3}, $V^\prime( x^v|_{\{v,w\}})=0$. Since $v\in \V_1$, for $w^\prime\sim w$ such that $w^\prime\neq v$ we know that $x^v_w=x^v_{w^\prime}=b\neq a$. Therefore by \eqref{equation:listinteract1} and \eqref{equation:listinteract2} $$V^\prime( x^v|_{\{w^\prime,w\}})=V( x^v|_{\{w^\prime,w\}})$$
and
$$V^\prime( x^v|_{\{w\}})=V( x^v|_{\{w\}}).$$
Since $\l a,a\r_{\{v,w\}}\in\B_{\{v,w\}}(X)$ therefore $v,w\in \V_1$ and $x^v_{w^\prime}=x^w_{w^\prime}=b$ for all $w^\prime\sim w$, $w^\prime\neq v$. Then by \eqref{equation:listinteract3}
\begin{eqnarray*}
V^\prime( \theta^w_a(x^v)|_{\{w^\prime,w\}})= V^\prime( \l b,a\r_{\{w^\prime,w\}})=V^\prime( x^w|_{\{w^\prime,w\}})=0.
\end{eqnarray*}
By \eqref{equation:listinteract4} we get that
$$V^\prime( \theta^w_a(x^v)|_{\{w\}})=V^\prime(\l a\r_{\{w\}})=M(\theta^w_b(x^w), x^w)+ V( \theta^w_b(x^w)|_{\{w\}})+\sum_{w^\prime: w^\prime\sim w} V(\theta^w_b(x^w)|_{\{w,w^\prime\}}).$$
Putting all this together, we get
\begin{eqnarray}
V^\prime(\l a,a\r_{\{v,w\}})&=&M(x^v, \theta^w_a(x^v))+\displaystyle{\sum_{w^\prime: w^\prime\sim w, w^\prime\neq v}V( x^v|_{\{w^\prime,w\}}) }+V( x^v|_{\{w\}})-M(\theta^w_b(x^w), x^w)\nonumber\\
&&- V( \theta^w_b(x^w)|_{\{w\}})-(\sum_{w^\prime: w^\prime\sim w} V( \theta^w_b(x^w)|_{\{w,w^\prime\}})).\label{equation:listinteract6}
\end{eqnarray}

This completes proof for uniqueness. It also follows from the arguments given above that an interaction $V^\prime$ satisfies \eqref{equation:listinteract1}, \eqref{equation:listinteract2}, \eqref{equation:listinteract3}, \eqref{equation:listinteract4}, \eqref{equation:listinteract5} and \eqref{equation:listinteract6} if and only if the pairs listed in (\ref{item:from_X_a}), (\ref{item:changing_b_at_v}), (\ref{item:changing_at_boundary_no_a}) and (\ref{item:changing_at_boundary_a}) are $V^\prime$-good.\\

Consider a nearest neighbour interaction $V^\prime$ on $X$ given by the following:
\begin{enumerate}[ (i)]
\item If $v\sim w\in \V$ and $\l c,d\r_{\{v,w\}}\in\B_{\{v,w\}}(X_a)$ then $V^\prime(\l c,d\r_{v,w})$ is given by \eqref{equation:listinteract1}
\item and $V^\prime(\l c\r_v)$ is given by \eqref{equation:listinteract2}.
\item If $v\in \V_1\cup \V_2$ and $w\sim v$, then $V^\prime( x^v|_{\{v,w\}})$ is given by \eqref{equation:listinteract3}.
\item If $v\in \V_1\cup \V_2$, the value $V^\prime(\l a\r_v)$ is given by \eqref{equation:listinteract4}.
\item If $v\in \V_1\cup \V_2$, $w\sim v$ and $c\in \A\setminus \{x^v_w,a\}$ such that $\l a,c\r_{\{v,w\}}\in \B_{\{v,w\}}(X)$, the value $V^\prime(\l a,c\r_{\{v,w\}})$ is given by \eqref{equation:listinteract5}.
\item If $v\in \V_1\cap P_1$, $w\sim v$ such that $\l a,a\r_{\{v,w\}}\in \B_{\{v,w\}}(X)$, the value $V^\prime(\l a,a\r_{\{v,w\}})$ is given by \eqref{equation:listinteract6}.
\end{enumerate}•
By the preceding paragraph the proof is complete.
\end{proof}

The following lemma explains why the weak conclusions of Lemma \ref{lemma:construction_of_V^prime} are sufficient and completes the proof of Theorem \ref{thm:main theorem}.

\begin{lemma}\label{lemma:from_particular_pairs_to_all_pairs}
Let $\G=(\V, \E1)$ be a bipartite graph with partite classes $P_1$ and $P_2$, $\A$ be a finite set, $X\subset \A^\V$ be an n.n.constraint space and $X_a$ be a strong config-fold of $X$ where the symbol $a$ is strongly config-folded into the symbol $b$. Consider $\V_1, \V_2\subset \V$ and $x^v\in X$ for all $v\in \V_1\cup \V_2$ satisfying the conclusion of Lemma \ref{lemma:four special configurations}. Let $M\in \M_X$ be a Markov cocycle on $X$ such that $M|_{\Delta_{X_a}}$ is a Gibbs cocycle with some nearest neighbour interaction $V$ and $V^\prime$ be an interaction on $X$ as obtained in Lemma \ref{lemma:construction_of_V^prime}. Then $M\in \Gi_X$ is Gibbs with some nearest neighbour interaction $V^\prime$.
\end{lemma}
\begin{proof}
We will use the $V^\prime$-good pairs guaranteed by Lemma \ref{lemma:construction_of_V^prime} as steps in proving the following pairs are $V^\prime$-good:
\begin{enumerate}[(a)]
\item Let $x\in X$ and $v\in \V_1\cup \V_2$ such that $x_v=a$ and $x_w\neq a$ for all $w\sim v$. Then $(x, \theta^v_b(x))$ is $V^\prime$-good.
\label{item:nearby_not_a}
\item Let $x\in X$ and $v\in \V_1\cap P_1$ and $w\sim v$ such that $x_v=x_w=a$. Then $(x, \theta^v_b(x))$ is $V^\prime$-good. \label{item:nearby_is_a}
\item All asymptotic pairs $(x,y)\in \Delta_X$ are $V^\prime$-good.
\label{item:no_constraint}
\end{enumerate}
Given an asymptotic pair, Statements (\ref{item:nearby_not_a}) and (\ref{item:nearby_is_a}) allow replacement of the $a$'s by $b$'s giving us a pair in $\Delta_{X_a}$. From Conclusion (\ref{item:from_X_a}) in Lemma \ref{lemma:construction_of_V^prime} we know that all pairs in $\Delta_{X_a}$ are $V^\prime$-good. Since the relation $V^\prime$-good is an equivalence relation this proves Statement (\ref{item:no_constraint}) thereby completing the proof. Now let us look at the details. An asymptotic pair $(x,y)\in \Delta_X$ is said to be \emph{Markov-similar to $(x',y')$ for a set $A\subset \V$} if $x|_{A^c}=y|_{A^c}$, $x'|_{A^c}=y'|_{A^c}$ and $x|_{A\cup \partial A}=x'|_{A\cup \partial A}$, $y|_{A\cup \partial A}=y'|_{A\cup \partial A}$.

Consider any $v\in \V_1\cup \V_2$ and $x\in X$ such that $x_v=a$. Let
\begin{equation*}
\partial\{v\}=\{w_1, w_2, \ldots w_n\}.
\end{equation*}
Since for all $1\leq r\leq n$, $[a,x_{w_r}]_{\{v,w_r\}}\in \B_{\{v,w_r\}}(X)$, Lemma \ref{lemma:four special configurations} implies that
$$\theta^{w_r,w_{r+1} \ldots, w_n}_{x_{w_r},x_{w_{r+1}},\ldots,x_{w_n}}(x^v)\in X. $$
Let $x^1=\theta^{w_1, w_2, \ldots, w_n}_{x_{w_1},x_{w_2},\ldots,x_{w_n}}(x^v)$. The pair $(x,\theta^v_b(x))$ is Markov-similar to $(x^1, \theta^v_b(x^1))$ for the set $\{v\}$. Note
$$\theta^{w_1, w_2, \ldots, w_{r}}_{x^{v}_{w_1},x^{v}_{w_2},\ldots,x^{v}_{w_{r}}}(x^1)=\theta^{w_{r+1} \ldots, w_n}_{x_{w_{r+1}},\ldots,x_{w_n}}(x^v)\in X$$
and that $x^1$ and $x^v$ differ only on $\partial\{v\}$. \\

In the following we will remove $a$'s in configurations from those vertices which are isolated from other $a$'s.

\noindent\emph{Proof of Statement (\ref{item:nearby_not_a}):} Consider the sequence
\begin{eqnarray*}
x^1,\theta^{w_1}_{x^v_{w_1}}(x^1),\theta^{w_1,w_2}_{x^v_{w_1},x^v_{w_2}}(x^1), \ldots, \theta^{w_1, w_2, \ldots, w_n}_{x^{v}_{w_1},x^{v}_{w_2},\ldots,x^{v}_{w_n}}(x^1)=x^v,\theta^v_b(x^v),\theta^v_b(x^1).
\end{eqnarray*}
Here single site changes have been made on $\partial \{v\}$ taking us from $x^1$ to $x^v$ and then the symbol at $v$ has been changed to obtain $\theta^v_b(x^v)$. In the last step $\theta^v_b(x^v)$ has been changed on $\partial \{v\}$ to obtain $\theta^v_b(x^1)$. 

\noindent Each pair
$$(\theta^{w_1, w_2, \ldots, w_r}_{x^{v}_{w_1},x^{v}_{w_2},\ldots,x^{v}_{w_r}}(x^1),\theta^{w_1, w_2, \ldots, w_{r+1}}_{x^{v}_{w_1},x^{v}_{w_2},\ldots,x^{v}_{w_{r+1}}}(x^1))$$
is Markov-similar to $(\theta^{w_{r+1}}_{x^1_{w_{r+1}}}(x^v),x^v)$ for the set $\{w_{r+1}\}$ for all $0\leq r\leq n-1$. By Conclusion (\ref{item:changing_at_boundary_no_a}) in Lemma \ref{lemma:construction_of_V^prime}, $(\theta^{w_{r+1}}_{x^1_{w_{r+1}}}(x^v),x^v)$ is $V^\prime$-good for all $0\leq r\leq n-1$. Thus by Corollary \ref{corollary:chain_of_V-goodness}, we get that $(x^1,x^v)$ is $V^\prime$-good. By Conclusion (\ref{item:changing_b_at_v}) in Lemma \ref{lemma:construction_of_V^prime} and symmetry of the relation $V^\prime$-good we get that $(x^v,\theta^v_b(x^v))$ is $V^\prime$-good. Since $\theta^v_b(x^v),\theta^v_b(x^1)\in X_a$, Conclusion (\ref{item:from_X_a}) in Lemma \ref{lemma:construction_of_V^prime} implies that $(\theta^v_b(x^v),\theta^v_b(x^1))$ is $V^\prime$-good. By Corollary \ref{corollary:chain_of_V-goodness} we arrive at $(x^1, \theta^v_b(x^1))$ being $V^\prime$-good. But $(x^1, \theta^v_b(x^1))$ is Markov-similar to $(x,\theta^v_b(x))$. Therefore by Proposition \ref{prop:V-goodness_equivalence_markov_similar} we get that $(x,\theta^v_b(x))$ is $V^\prime$-good.
\bigskip

In the next step we remove the $a$'s which are not isolated.

\noindent\emph{Proof of Statement (\ref{item:nearby_is_a}):} We construct a sequence from $x$ to $\theta^v_b(x)$ in three parts. In the first part single site changes will be made on $\partial \{v\}$ taking us from $x^1$ to $x^v$. In the second part the symbol at $v$ will be changed to obtain $\theta^v_b(x^v)$. In the last part single site changes will be made on $\partial \{v\}$ to obtain $\theta^v_b(x^1)$ from $\theta^v_b(x^v)$.

Consider the sequence
\begin{gather*}
(x^1,\theta^{w_1}_{x^v_{w_1}}(x^1),\theta^{w_1,w_2}_{x^v_{w_1},x^v_{w_2}}(x^1), \ldots, \theta^{w_1, w_2, \ldots, w_n}_{x^{v}_{w_1},x^{v}_{w_2},\ldots,x^{v}_{w_n}}(x^1)=x^v),\\
(x^v,\theta^v_b(x^v)),\\
(\theta^v_b(x^v),\theta^{w_1}_{x^1_{w_1}}(\theta^v_b(x^v)),\theta^{w_1,w_2}_{x^1_{w_1},x^1_{w_2}}(\theta^v_b(x^v)), \ldots, \theta^{w_1, w_2, \ldots, w_n}_{x^{1}_{w_1},x^{1}_{w_2},\ldots,x^{1}_{w_n}}(\theta^v_b(x^v))=\theta^v_b(x^1)).
\end{gather*}

In the first part of the sequence notice that
$$(\theta^{w_1, w_2, \ldots, w_r}_{x^{v}_{w_1},x^{v}_{w_2},\ldots,x^{v}_{w_r}}(x^1),\theta^{w_1, w_2, \ldots, w_{r+1}}_{x^{v}_{w_1},x^{v}_{w_2},\ldots,x^{v}_{w_{r+1}}}(x^1))$$
is Markov-similar to $(\theta^{w_{r+1}}_{x^1_{w_{r+1}}}(x^v),x^v)$ for the set $\{w_{r+1}\}$ for all $0\leq r\leq n-1$. If for some $0\leq r\leq n-1$, $x^1_{w_{r+1}}\neq a$ then by Conclusion (\ref{item:changing_at_boundary_no_a}) in Lemma \ref{lemma:construction_of_V^prime} we get that $(\theta^{w_{r+1}}_{x^1_{w_{r+1}}}(x^v),x^v)$ is $V^\prime$-good. If for some $0\leq r\leq n-1$, $x^1_{w_{r+1}}= a$ then by Conclusion (\ref{item:changing_at_boundary_a}) in Lemma \ref{lemma:construction_of_V^prime} we get that $(\theta^{w_{r+1}}_{x^1_{w_{r+1}}}(x^v),x^v)$ is $V^\prime$-good. Proposition \ref{prop:V-goodness_equivalence_markov_similar} implies that
$$(\theta^{w_1, w_2, \ldots, w_r}_{x^{v}_{w_1},x^{v}_{w_2},\ldots,x^{v}_{w_r}}(x^1),\theta^{w_1, w_2, \ldots, w_{r+1}}_{x^{v}_{w_1},x^{v}_{w_2},\ldots,x^{v}_{w_{r+1}}}(x^1))$$
is $V^\prime$-good for all $0\leq r\leq n-1$. By Corollary \ref{corollary:chain_of_V-goodness}, we get that $(x^1,x^v)$ is $V^\prime$-good and we are done with the first part of the sequence.

For the second part of the sequence by Conclusion (\ref{item:changing_b_at_v}) in Lemma \ref{lemma:construction_of_V^prime} and by the symmetry of the relation $V^\prime$-good we get that $(x^v,\theta^v_b(x^v))$ is $V^\prime$-good.

For the third part of the sequence the asymptotic pair
$$(\theta^{w_1, w_2, \ldots, w_r}_{x^{1}_{w_1},x^{1}_{w_2},\ldots,x^{1}_{w_r}}(\theta^v_b(x^v)),\theta^{w_1, w_2, \ldots, w_{r+1}}_{x^{1}_{w_1},x^{1}_{w_2},\ldots,x^{1}_{w_{r+1}}}(\theta^v_b(x^v)))$$
is Markov-similar to $(\theta^v_b(x^v),\theta^{w_{r+1}}_{x^1_{w_{r+1}}}(\theta^v_b(x^v)))$ for the set $\{w_{r+1}\}$ for all $0\leq r\leq n-1$.

If for some $0\leq r\leq n-1$, $x^1_{w_{r+1}}\neq a$ then $(\theta^v_b(x^v),\theta^{w_{r+1}}_{x^1_{w_{r+1}}}(\theta^v_b(x^v))\in X_a$ and by Conclusion (\ref{item:from_X_a}) in Lemma \ref{lemma:construction_of_V^prime} we get that $(\theta^v_b(x^v),\theta^{w_{r+1}}_{x^1_{w_{r+1}}}(\theta^v_b(x^v)))$ is $V^\prime$-good. Since $v\in \V_1$, $x^v|_{D_2(v)\setminus\{v\}}=b$. Thus if for some $0\leq r\leq n-1$, $x^1_{w_{r+1}}=a$ then
$$(\theta^{w_{r+1}}_{x^1_{w_{r+1}}}(\theta^v_b(x^v)),\theta^v_b(x^v))=(\theta^{w_{r+1}}_{a}(\theta^v_b(x^v)),\theta^{w_{r+1}}_b(\theta^{w_{r+1}}_{a}(\theta^v_b(x^v)))$$
and $(\theta^{w_{r+1}}_{a}(\theta^v_b(x^v)))_{w^\prime}=b\neq a$ for all $w^\prime\sim w_{r+1}$. By Statement (\ref{item:nearby_not_a}) in the proof of this lemma we get that $(\theta^{w_{r+1}}_{a}(\theta^v_b(x^v)),\theta^{w_{r+1}}_b(\theta^{w_{r+1}}_{a}(\theta^v_b(x^v)))$ is $V^\prime$-good. By symmetry of the relation $V^\prime$-good we get that $(\theta^v_b(x^v),\theta^{w_{r+1}}_{x^1_{w_{r+1}}}(\theta^v_b(x^v)))$ is $V^\prime$-good in this case as well.

Thus for all $0\leq r \leq n-1$ we find that $(\theta^v_b(x^v),\theta^{w_{r+1}}_{x^1_{w_{r+1}}}(\theta^v_b(x^v)))$ is $V^\prime$-good. Using Corollary \ref{corollary:chain_of_V-goodness} we find that $(\theta^v_b(x^v), \theta^v_b(x^1))$ is $V^\prime$-good.

So we have proven that $(x^1, x^v), (x^v, \theta^v_b(x^v)), (\theta^v_b(x^v), \theta^v_b(x^1))$ are $V^\prime$-good. Stringing them by Corollary \ref{corollary:chain_of_V-goodness} we get that $(x^1,\theta^v_b(x^1))$ is $V^\prime$-good. But $(x^1, \theta^v_b(x^1))$ is Markov-similar to $(x,\theta^v_b(x))$. Therefore by Proposition \ref{prop:V-goodness_equivalence_markov_similar} we get that $(x,\theta^v_b(x))$ is $V^\prime$-good.\\

The previous two statements give us the freedom to change the $a$'s into $b$'s. Now we will use them to prove the last statement.

\noindent\emph{Proof of Statement (\ref{item:no_constraint}):} Consider an asymptotic pair $(x,y)\in \Delta_X$. Let
$$F:=\{v\in \V\:|\:x_v\neq y_v\}$$
and $x^1, y^1\in \A^\V$ be obtained by replacing the $a$'s outside $F\cup \partial F$ by $b$'s, that is
$$x^1_u:=\begin{cases}
x_u&\text{ if }u \in F\cup\partial F \text{ or }x_u\neq a\\
b&\text{ otherwise}
\end{cases}•$$
and
$$y^1_u:=\begin{cases}
y_u&\text{ if }u \in F\cup\partial F \text{ or }y_u\neq a\\
b&\text{ otherwise}.
\end{cases}•$$
By \eqref{equation:foldingdefn1} $x^1, y^1\in X$. Since $x=y$ on $F^c$, it follows that $(x,y)$ and $(x^1, y^1)$ are Markov-similar but there are only finitely many vertices where $x^1$ and $y^1$ equal $a$. Let
\begin{eqnarray*}
\{v_1, v_2,\ldots, v_r\}&:=&\{v\in P_1\:|\: x^1_v=a\}\\
\{w_1, w_2,\ldots, w_{r^\prime}\}&:=&\{w\in P_1\:|\: y^1_w=a\}\\
\{v_{r+1}, v_{r+2} \ldots, v_{r+k}\}&:=&\{v\in P_2\:|\: x^1_v=a\}\\
\{w_{r^\prime+1}, w_{r^\p+2} \ldots, w_{r^\p+k^\p}\}&:=&\{w\in P_2\:|\: y^1_w=a\}
\end{eqnarray*}
index the vertices with $a$ in $x^1$ and $y^1$. By Lemma \ref{lemma:four special configurations} the configurations $\theta^{v_1, v_2 \ldots, v_{i}}_{b,b, \ldots, b}(x^1)$ and $\theta^{w_1, w_2 \ldots, w_{i^\prime}}_{b,b, \ldots, b}(y^1)$ are configurations in $X$ for all $1\leq i \leq r+k$ and $1\leq i^\prime \leq r^\prime+k^\prime$. Therefore we can consider the Sequences (\ref{sequence:sequence_in_final_1} to \ref{sequence:sequence_in_final_5}) given below:

\noindent We begin by replacing the $a$'s in $x^1$ from the partite class $P_1$ by $b$'s.
\begin{gather}
(x^1, \theta^{v_1}_b(x^1),\theta^{v_1, v_2}_{b,b}(x^1),\ldots,\theta^{v_1, v_2,\ldots, v_r}_{b,b,\ldots,b}(x^1)).\label{sequence:sequence_in_final_1}
\end{gather}
In the resulting configuration $\theta^{v_1, v_2,\ldots, v_r}_{b,b,\ldots,b}(x^1)$ adjacent vertices cannot both have the symbol $a$; the $a$'s left in the configuration $x^1$ are changed to $b$'s.
\begin{gather}
(\theta^{v_1, v_2,\ldots, v_r}_{b,b,\ldots,b}(x^1),\theta^{v_1, v_2,\ldots, v_{r+1}}_{b,b,\ldots,b}(x^1),\ldots,\theta^{v_1, v_2,\ldots, v_{r+k}}_{b,b,\ldots,b}(x^1)).\label{sequence:sequence_in_final_2}
\end{gather}
After removing the $a$'s from $x^1$ and $y^1$ the configurations obtained are elements of $X_a$. 
\begin{gather}
(\theta^{v_1, v_2,\ldots, v_{r+k}}_{b,b,\ldots,b}(x^1),\theta^{w_1, w_2 \ldots, w_{r^\prime+k^\prime}}_{b,b, \ldots, b}(y^1)).\label{sequence:sequence_in_final_3}
\end{gather}
The changes occurring in Sequences \ref{sequence:sequence_in_final_1} and \ref{sequence:sequence_in_final_2} are used in reverse to obtain $y^1$ starting with $\theta^{w_1, w_2 \ldots, w_{r^\prime+k^\prime}}_{b,b, \ldots, b}(y^1)$.
\begin{gather}
(\theta^{w_1, w_2 \ldots, w_{r^\prime+k^\prime}}_{b,b, \ldots, b}(y^1),\theta^{w_1, w_2 \ldots, w_{r^\prime+k^\prime-1}}_{b,b, \ldots, b}(y^1), \ldots, \theta^{w_1, w_2 \ldots, w_{r^\prime}}_{b,b, \ldots, b}(y^1)),\label{sequence:sequence_in_final_4}\\
(\theta^{w_1, w_2 \ldots, w_{r^\prime}}_{b,b, \ldots, b}(y^1),\theta^{w_1, w_2 \ldots, w_{r^\prime-1}}_{b,b, \ldots, b}(y^1),\ldots, \theta^{w_1}_b(y^1),y^1)\label{sequence:sequence_in_final_5}.
\end{gather}

For all $1\leq i \leq r$, the vertex $v_i\in P_1$ and the symbol $x^1_{v_i}=a$. Thus by Statement (\ref{item:nearby_not_a}) and (\ref{item:nearby_is_a}) in this proof we get that
$$(\theta^{v_1, v_2,\ldots, v_i}_{b,b,\ldots,b}(x^1),(\theta^{v_1, v_2,\ldots, v_{i+1}}_{b,b,\ldots,b}(x^1))$$
is $V^\prime$-good. Thus all adjacent pairs in the Sequence \ref{sequence:sequence_in_final_1} are $V^\prime$-good

Notice that $(\theta^{v_1, v_2, \ldots, v_r}_{b,b,\ldots, b}(x^1))_v\neq a$ for all $v\in P_1$ and hence $(\theta^{v_1, v_2,\ldots, v_{r+i}}_{b,b,\ldots,b}(x^1))_v\neq a$ for all $1\leq i \leq k$ and $v\in P_1$. Now consider an adjacent pair in the Sequence \ref{sequence:sequence_in_final_2},
$$(\theta^{v_1, v_2,\ldots, v_{r+i}}_{b,b,\ldots,b}(x^1),\theta^{v_1, v_2,\ldots, v_{r+i+1}}_{b,b,\ldots,b}(x^1))$$
for some $0\leq i \leq k-1$. Since $v_{r+i+1}\in P_2$ , $(\theta^{v_1, v_2,\ldots, v_{r+i}}_{b,b,\ldots,b}(x^1))_w\neq a$ for all $w\sim v_{r+i+1}$. But $(\theta^{v_1, v_2,\ldots, v_{r+i}}_{b,b,\ldots,b}(x^1))_{v_{r+i+1}}=a$, therefore by Statement (\ref{item:nearby_not_a}) we get that
$$(\theta^{v_1, v_2,\ldots, v_{r+i}}_{b,b,\ldots,b}(x^1),\theta^{v_1, v_2,\ldots, v_{r+i+1}}_{b,b,\ldots,b}(x^1))$$
is $V^\prime$-good.

Notice that $(\theta^{v_1, v_2, \ldots, v_{r+k}}_{b,b,\ldots, b}(x^1)),(\theta^{w_1, w_2, \ldots, w_{r^\p+k^\p}}_{b,b,\ldots, b}(y^1))\in X_a$. Thus by Conclusion (\ref{item:from_X_a}) in Lemma \ref{lemma:construction_of_V^prime}, we get that the Pair \ref{sequence:sequence_in_final_3} is $V^\prime$-good.

The proof that the adjacent pairs listed in Sequences \ref{sequence:sequence_in_final_4} and \ref{sequence:sequence_in_final_5} are $V^\prime$-good is identical to the proof for the Sequences \ref{sequence:sequence_in_final_2} and \ref{sequence:sequence_in_final_1} with an additional use of the symmetry of the relation $V^\prime$-good.

Thus all adjacent pairs in Sequences \ref{sequence:sequence_in_final_1}-\ref{sequence:sequence_in_final_5} are $V^\prime$-good. By Corollary \ref{corollary:chain_of_V-goodness} we get that $(x^1,y^1)$ is $V^\prime$-good. But $(x^1,y^1)$ is Markov-similar to $(x,y)$. By Proposition \ref{prop:V-goodness_equivalence_markov_similar} we have that $(x,y)$ is $V^\prime$-good. This completes the proof.
\end{proof}

If $\G$ is finite then Theorem \ref{thm:G-invariant main theorem} follows immediately from Theorem \ref{thm:main theorem}: if $V^\prime$ is a nearest neighbour interaction for a $Gr$-invariant Markov cocycle $M$ then 
\begin{equation*}
\frac{\sum_{g\in Gr} gV^\prime}{|Gr|}
\end{equation*}
is a $Gr$-invariant nearest neighbour interaction for $M$. We will prove the following result which along with Proposition \ref{proposition:induced_map_cocycles} immediately implies Theorem \ref{thm:G-invariant main theorem}.

\begin{thm}\label{thm:extension of a Gibbs cocycle}
Let $\G=(\V, \E1)$ be a bipartite graph and $\A$ be a finite alphabet. Let $Gr\subset Aut(\G)$ be a subgroup, $X\subset \A^\V$ be a $Gr$-invariant n.n.constraint space and $X_a$ be a strong config-fold of $X$ for some $a\in \A$. Suppose $M\in \M^{Gr}_{X}$ is a $Gr$-invariant Markov cocycle. Then $M\in \Gi^{Gr}_{X}$ if and only if $M|_{\Delta_{X_a}}\in \Gi^{Gr}_{X_a}$.
\end{thm}

\begin{proof} By Proposition \ref{proposition:induced_map_cocycles}, $M\in \Gi^{Gr}_{X}$ implies $M|_{\Delta_{X_a}}\in \Gi^{Gr}_{X_a}$. We will prove the converse. Let $M\in \M^{Gr}_X$ such that $M|_{\Delta_{X_a}}\in \Gi^{Gr}_{X_a}$. Let $V$ be a $Gr$-invariant nearest neighbour interaction for $M|_{\Delta_{X_a}}$.

Mimicking the proof of Lemma \ref{lemma:four special configurations} we will now obtain special configurations $x^v$ in a $Gr$-invariant way.

\begin{lemma}\label{lemma:G_invariant_four_special_configurations}
Let $\G=(\V, \E1)$ be a bipartite graph and $\A$ be a finite alphabet. Let $Gr\subset Aut(\G)$ be a subgroup, $X\subset \A^\V$ be a $Gr$-invariant n.n.constraint space on $\G$ and $X_a$ be a strong config-fold of $X$ for some $a\in \A$. Let
$$\V_1:=\{v\in \V\:|\:\text{ there exists }w\sim v\text{ such that }\l a,a\r_{\{v,w\}}\in \B_{\{v,w\}}(X)\}$$
and
$$\V_2:=\{v\in \V\setminus \V_1\:|\:\l a\r_{v}\in \B_{\{v\}}(X)\}.$$
Then $\V_1$ and $\V_2$ are invariant under the action of $Gr$. Moreover for all $v\in \V_1\cup \V_2$ there exists $x^v\in X$ satisfying the conclusions of Lemma \ref{lemma:four special configurations} such that $(gx^v)|_{gD_2(v)}=x^{gv}|_{gD_2(v)}$ for all $g\in Gr$.
\end{lemma}

\begin{proof}
Since $X$ is $Gr$-invariant it follows that the sets $\V_1$ and $\V_2$ are $Gr$-invariant.  Let the symbol $a$ be strongly config-folded into the symbol $b$.

Let $v\in \V_1$ and $g\in Gr$. Then by Lemma \ref{lemma:four special configurations} there exists $x^v, x^{gv}\in X$ such that $x^v_v=x^{gv}_{gv}=a$ and $x^v_u=x^{gv}_{gu}=b$ for all $u \in D_2(v)\setminus \{v\}$. Thus we find that $(gx^v)|_{gD_2(v)}=x^{gv}|_{gD_2(v)}$.

Let $v\in \V_2$. Then for all $w\sim v$, $g\in Gr$ and $c\in\A\setminus\{a\}$ the pattern $\l a,c\r_{\{v,w\}}\in \B_{\{v,w\}}(X)$ if and only if $\l a,c_{v,w}\r_{\{gv,gw\}}\in \B_{\{gv,gw\}}(X)$. Thus for all $w\sim v$ we can choose $c_{v,w}\in \A\setminus\{a\}$ such that $\l a,c_{v,w}\r_{\{v,w\}}\in \B_{\{v,w\}}(X)$ and $c_{v,w}= c_{gv,gw}$ for all $g\in Gr$. Note that since $v\in \V_2$ we know that $c_{v,w}\neq a$.

By \eqref{equation:foldingdefn3} there exists $x^{1,v}\in X$ such that $x^{1,v}|_{\partial D_1(v)}=b^{\partial D_1(v)}$. Since $a$ can be strongly config-folded into the symbol $b$ we can assume that $x^{1,v}\in X_a$. Consider $x^v\in \A^\V$ defined by
\begin{equation*}
x^v_u:=\begin{cases}
a&\text{ if }u=v\\
c_{v,u}&\text{ if } u\sim v\\
x^{1,v}_u&\text{ if } u\in D_1(v)^c.
\end{cases}•
\end{equation*}•
Note that $a$ appears in $x^{v}$ only at the vertex $v$. Any edge $(u_1, u_2)$ in $\G$ lies either completely in $D_2(v)$ or in $D_1(v)^c$. If the edge lies in $D_1(v)^c$ then $x^v|_{\{u_1, u_2\}}= x^{1,v}|_{\{u_1, u_2\}}\in \B_{\{u_1, u_2\}}(X)$. If the edge is of the form $(v,w)$ then $x^v|_{\{v, w\}}= \l a,c_{v,w}\r_{\{v,w\}}\in \B_{\{v, w\}}(X)$. If the edge is of the form $(w, w^\prime)$ where $w\in \partial \{v\}$ and $w^\prime\in \partial D_1(v)$ then $x^v|_{\{w, w^\prime\}}= \l c_{v,w},b\r_{\{w,w^\p\}}$. Since $(v,w)$ and $(w, w^\prime)$ are edges in the graph $\G$ and $\l a,c\r_{\{v,w\}}\in \B_{\{v,w\}}(X)$ by \eqref{equation:foldingdefn2} we know that $\l c_{v,w},b\r_{\{w,w^\p\}}\in \B_{\{w,w^\p\}}(X)$.

Thus we have proved for every edge $(u_1,u_2)$ in $\G$ that $x^v|_{\{u_1,u_2\}}\in \B_{\{u_1,u_2\}}(X)$. Since $X$ is an n.n.constraint space we get that $x^v\in X$.

Moreover for all $v\in \V_2$ and $g\in Gr$
\begin{equation*}
(gx^v)_{u}=\begin{cases}
a& \text{ if }u=gv\\
c_{v,g^{-1}u}&\text{ if } u\sim gv\\
b&\text{ if } u\in \partial D_1(gv)
\end{cases}•
\end{equation*}•
and
\begin{equation*}
(x^{gv})_{u}=\begin{cases}
a&\text{ if }u=gv\\
c_{gv,u}=c_{v,g^{-1}u}&\text{ if } u\sim gv\\
b&\text{ if } u\in \partial D_1(gv),
\end{cases}•
\end{equation*}•
that is, $(gx^v)|_{gD_2(v)}=x^{gv}|_{gD_2(v)}$.

Thus the configurations $x^v$ satisfy the Conclusions (\ref{enumerate:configuration1}) and (\ref{enumerate:configuration2}) of Lemma \ref{lemma:four special configurations} and $(gx^v)|_{gD_2(v)}=x^{gv}|_{gD_2(v)}$ for all $g\in Gr$ and $v\in \V_1\cup\V_2$. The rest follows exactly as in the proof of Lemma \ref{lemma:four special configurations}.
\end{proof}
Consider sets $\V_1, \V_2\subset \V$ and for all $v\in \V$ configurations $x^v\in X$ as obtained by Lemma \ref{lemma:G_invariant_four_special_configurations}. Then by Lemma \ref{lemma:construction_of_V^prime} there exists a unique nearest neighbour interaction $V^\prime$ on $X$ such that the Pairs (\ref{item:from_X_a}), (\ref{item:changing_b_at_v}), (\ref{item:changing_at_boundary_no_a}) and (\ref{item:changing_at_boundary_a}) listed in Lemma \ref{lemma:construction_of_V^prime} are $V^\prime$-good. By Lemma \ref{lemma:from_particular_pairs_to_all_pairs} we get that $V^\prime$ is a nearest neighbour interaction for $M$. We will prove that the interaction $V^\prime$ is $Gr$-invariant. For this we will invoke the uniqueness of the interaction satisfying the conclusions of Lemma \ref{lemma:construction_of_V^prime}.

Let $g\in Gr$. $gV^\prime$ is a nearest neighbour interaction corresponding to $gM=M$. Thus the pairs listed in (\ref{item:from_X_a}), (\ref{item:changing_b_at_v}), (\ref{item:changing_at_boundary_no_a}) and (\ref{item:changing_at_boundary_a}) in Lemma \ref{lemma:construction_of_V^prime} are $gV^\prime$-good. Since $V$ is $Gr$-invariant, $gV^\prime|_{\B_{X_a}}=gV=V$. Hence $gV^\prime$ satisfies \eqref{equation:listinteract1} and \eqref{equation:listinteract2}.

If $v\in \V_1\cup \V_2$ then we know from Lemma \ref{lemma:G_invariant_four_special_configurations} that $(gx^v)|_{gD_1(v)}=x^{gv}|_{gD_1(v)}$. Since $V^\prime$ satisfies \eqref{equation:listinteract3}, for $w\sim v$ we get that
\begin{eqnarray*}
gV^\prime(x^v|_{\{v,w\}})=V^\prime(\l x^v_{v}, x^v_w\r_{\{g^{-1}v,g^{-1}w\}})=V^\prime(\l x^{g^{-1}v}_{g^{-1}v}, x^{g^{-1}v}_{g^{-1}w}\r_{\{g^{-1}v,g^{-1}w\}})=0.
\end{eqnarray*}
Thus the interaction $gV^\prime$ satisfies \eqref{equation:listinteract3}. Therefore the interaction $gV^\prime$ is a nearest neighbour interaction which satisfies \eqref{equation:listinteract1}, \eqref{equation:listinteract2} and \eqref{equation:listinteract3} such that the pairs listed in (\ref{item:from_X_a}), (\ref{item:changing_b_at_v}), (\ref{item:changing_at_boundary_no_a}) and (\ref{item:changing_at_boundary_a}) in Lemma \ref{lemma:construction_of_V^prime} are $gV^\prime$-good. By Lemma \ref{lemma:construction_of_V^prime} we know that such an interaction is unique. Thus $gV^\prime=V^\prime$ and $M\in \Gi^{Gr}_X$.

\end{proof}

\begin{corollary}\label{corollary:folds_quotient_isomorphic}
Let $\G=(\V, \E1)$ be a bipartite graph and $\A$ a finite alphabet. Let $Gr\subset Aut(\G)$ be a subgroup and $X\subset \A^\V$ be a $Gr$-invariant n.n.constraint space, $X_a$ be a strong config-fold of $X$ for some $a\in \A$. Then $\M^{Gr}_X/\Gi^{Gr}_X$ is isomorphic to $\M^{Gr}_{X_a}/\Gi^{Gr}_{X_a}$.
\end{corollary}

Clearly this corollary subsumes Theorem \ref{thm:G-invariant main theorem}. Thereby to understand the difference between Markov and Gibbs cocycles it is sufficient to study the cocycles over configuration spaces which cannot be strongly config-folded any further.

Also this corollary is most relevant when the dimension of the quotient space $\M^{Gr}_{X}/\Gi^{Gr}_{X}$ is finite. This holds in the following two situations:
\begin{enumerate}
\item
The underlying graph $\G$ is finite.
\item
The underlying graph $\G$ is $\Z^d$ for some dimension $d$, $Gr$ is the group of translations on $\Z^d$ and the space $X$ has the generalised pivot property (defined in \cite[Section 3]{chandgotiameyerovitch}.
\end{enumerate}•

\begin{proof}
By Proposition \ref{proposition:induced_map_cocycles} the map $F:\M^{Gr}_X\longrightarrow \M^{Gr}_{X_a}$ given by
$$F(M):= M|_{\Delta_{X_a}}\text{ for all }M\in \M^{Gr}_{X_a}$$
is surjective. By Theorem \ref{thm:extension of a Gibbs cocycle} we know that for a Markov cocycle $M\in \M^{Gr}_{X}$, $M\in \Gi^{Gr}_{X}$ if and only if $M|_{\Delta_{X_a}}\in \Gi^{Gr}_{X_a}$. Thus $F^{-1}(\Gi^{Gr}_{X_a})=\Gi^{Gr}_{X}$.

Via the second isomorphism theorem for vector spaces the map
$$\tilde{F}:\M^{Gr}_X/F^{-1}(\Gi^{Gr}_{X_a})\longrightarrow \M^{Gr}_{X_a}/\Gi^{Gr}_{X_a}$$
given by
$$\tilde{F}(M\!\!\!\mod F^{-1}(\Gi^{Gr}_{X_a})):=F(M)\!\!\!\mod\Gi^{Gr}_{X_a} $$
is an isomorphism. Since $F^{-1}(\Gi^{Gr}_{X_a})=\Gi^{Gr}_{X}$ the proof is complete.
\end{proof}

\section{Further Directions}\label{section:conclude}
\begin{enumerate}
\item
By Theorems \ref{thm:main theorem} and \ref{thm:G-invariant main theorem} we have generalised the Hammersley-Clifford Theorem, but only when the graph $\G$ is bipartite. Can this be generalised further beyond the bipartite case? Note that $\G$ being bipartite is used at many critical parts of the proof e.g. the construction of the configurations $x^v$ in Lemma \ref{lemma:four special configurations}, construction of the interaction in Lemma \ref{lemma:construction_of_V^prime} etc.

\item

Suppose a finite graph $\H$ can be folded into a single vertex (with or without a self-loop) or an edge. We have proved that for any bipartite graph $\G$ the space $Hom(\G, \H)$ is Hammersley-Clifford. Also we have shown that being Hammersley-Clifford is invariant under foldings and unfoldings of $\H$. Following \cite{brightwell2000gibbs} we will call a graph $\H$ \emph{stiff} if it cannot be folded anymore. Fixing a particular domain graph say $\Z^2$, is it possible to classify all stiff graphs $\H$ for which $Hom(\Z^2,\H)$ is Hammersley-Clifford? What if we want $Hom(\G, \H)$ to be Hammersley-Clifford for all bipartite graphs $\G$?

\item

Suppose $\G$ is a finite bipartite graph. Then the space of cocycles $Hom(\G, \H)$ is finite dimensional for all finite graphs $\H$. Is there an efficient algorithm to determine the dimension of the space of cocycles $\M_{Hom(\G, \H)}$ and $\Gi_{Hom(\G, \H)}$?

In another direction, is there a graph $\H$ for which the space of $\sigma$-invariant Markov cocycles $\M^\sigma_{Hom(\Z^d, \H)}$ is infinite dimensional? Is there an algorithm to determine the dimension of $\M^\sigma_{Hom(\Z^d, \H)}$ and $\Gi^\sigma_{Hom(\Z^d, \H)}$? We know from \cite{chandgotiameyerovitch} that if $Hom(\Z^d, \H)$ has the pivot property then the dimension of  $\M^\sigma_{Hom(\Z^d, \H)}$ is finite, however we do not know the corresponding dimension beyond a few specific cases \cite[Section 5]{chandgotiameyerovitch}.
\end{enumerate}•

\section{Acknowledgments}
I would like to thank Professor Tom Meyerovitch for suggesting the problem and Professor Ronnie Pavlov for suggesting the possible extension of the result to the case when the underlying graph $\G$ is bipartite. I would also like to thank Professor Brian Marcus for many discussions and for wading through many versions of this paper and Professor Peter Winkler for hosting me and discussing various aspects of this result. I appreciate the suggestions given by the referee which largely helped improve the quality of the paper. The overarching influence of AA and in particular Quatro Amigos was extremely stimulating. This research was partly funded by the Four-Year Fellowship at the University of British Columbia.

\bibliographystyle{plain}
\bibliography{GEhammersley}
\end{document}